\documentclass[11pt, oneside]{amsart}
\RequirePackage{amsmath}
\usepackage{amsthm}

\usepackage[utf8]{inputenc}
\usepackage{hyperref}

\usepackage{nomencl}
\makenomenclature

\usepackage{mathtools}
\usepackage{enumerate}
\usepackage{cite}
\usepackage[shortlabels]{enumitem}
\usepackage{tikz-cd}
\usepackage{stmaryrd}

\usepackage{latexsym}
\usepackage{amsmath,amssymb,amsthm,amscd}
\usepackage{color}
\usepackage[nolabel, final]{showlabels}
\title{Finiteness of analytic cohomology of Lubin-Tate $(\varphi_L,\Gamma_L)$-modules}
\author{Rustam Steingart}
\address{Ruprecht-Karls-Universität Heidelberg,
	Mathematisches Institut,Im Neuenheimer Feld 205, D-69120 Heidelberg}
\email{rsteingart@mathi.uni-heidelberg.de}
\date{\today}

\theoremstyle{plain}
\newtheorem{thm}{Theorem}[subsection]
\newtheorem{lem}[thm]{Lemma}
\newtheorem{rem}[thm]{Remark}
\newtheorem{prop}[thm]{Proposition}
\newtheorem{cor}[thm]{Corollary}
\newtheorem*{cor*}{Corollary}
\newtheorem{introtheorem}{Theorem}
\theoremstyle{definition}
\newtheorem{defn}[thm]{Definition}

\newcommand{\NN}{\mathbb{N}}
\newcommand{\N}{\mathbb{N}}
\newcommand{\Gal}{\operatorname{Gal}}

\newcommand{\ZZ}{\mathbb{Z}}

\newcommand{\QQ}{\mathbb{Q}}
\newcommand{\Q}{\mathbb{Q}}

\newcommand{\CC}{\mathbb{C}}
\newcommand{\C}{\mathbb{C}}

\newcommand{\cR}{\mathcal{R}}

\newcommand{\cEnd}{\mathcal{E}nd}

\newcommand{\id}{\operatorname{id}}
\newcommand{\res}{\operatorname{res}}

\DeclarePairedDelimiter\abs{\lvert}{\rvert}%
\DeclarePairedDelimiter\norm{\lVert}{\rVert}%
\begin{document}
	\maketitle
\begin{abstract}
	We prove finiteness and base change properties for analytic cohomology of families of $L$-analytic $(\varphi_L,\Gamma_L)$-modules parametrised by affinoid algebras in the sense of Tate. For technical reasons we work over a field $K$ containing a period of the Lubin-Tate group, which allows us to describe analytic cohomology in terms of an explicit generalised Herr complex.
\end{abstract}
	\section*{Introduction}
		Let $L/\Q_p$ be a finite extension and let $\varphi_L$ be a Frobenius power series for a uniformiser $\pi_L.$ We denote by $L_\infty$ the Lubin-Tate extension of $L$ attached to $\varphi_L$ and set $\Gamma_L= \operatorname{Gal}(L_{\infty}/L).$ In the cyclotomic case, i.e., $L=\Q_p$ and $\varphi(T) = (1+T)^p-1$ there is, due to Fontaine, an equivalence of categories between $p$-adic Galois representations and étale $(\varphi,\Gamma)$-modules. The theory of $(\varphi,\Gamma)$-modules allows for rather simple descriptions of the Galois and Iwasawa cohomology of a representation $V$ in terms of the $(\varphi,\Gamma)$-module attached to $V$ (cf.\ \cite{herr1998cohomologie, cherbonnier1999theorie}) and similar results can be achieved for the Lubin-Tate $(\varphi_L,\Gamma_L)$-modules attached to $L$-linear representations of $G_L$ (cf.\ \cite{kupferer2022herr, SV15}). In order to describe invariants of an $L$-linear representation $V$ in terms of the $(\varphi_L,\Gamma_L)$-modules one would like to work with $(\varphi_L,\Gamma_L)$-modules over the Robba ring $\cR_L = \bigcup_{r \in[0,1)} \cR_L^{[r,1)},$ where $\cR_L^{[r,1)}$ denotes the ring of formal Laurent series with coefficients in $L$, which converge on the half open annulus defined by $[r,1).$ 
It turns out that the category of étale $(\varphi_L,\Gamma_L)$-modules over $\cR_L$ is equivalent to the category of overconvergent representations. While in the cyclotomic case every Galois representation is overconvergent (cf. \cite{cherbonnier1998representations}), the theory breaks down if  $L \neq \Q_p.$ A sufficient condition for overconvergence is analyticity  and a theorem of Berger asserts that there is an equivalence of categories between $L$-analytic representations and $L$-analytic $(\varphi_L,\Gamma_L)$-modules (cf. \cite{Berger2016}).
An $L$-linear representation is called $L$-analytic if $\CC_p \otimes_{L,\sigma} V$ is trivial as a $\C_p$-semi-linear representation for every non-identity embedding $\sigma\colon L  \to \C_p.$ A $(\varphi_L,\Gamma_L)$-module over $\cR_L$ is analytic if the action of $\Gamma_L$ is $L$-analytic, more precisely pro-$L$-analytic in the sense of \cite{Berger2016}. This is equivalent to requiring that the action is differentiable and the action of $\operatorname{Lie(\Gamma_L)}$ is $L$-bilinear. 
This leads to another interesting cohomology theory: the pro-analytic cohomology computed by the total complex of the double complex $$C^{\bullet}_\text{pro-an}(\Gamma_L,M) \xrightarrow{\varphi_L-1}C^{\bullet}_\text{pro-an}(\Gamma_L,M),$$ where $C^{\bullet}_\text{pro-an}(\Gamma_L,M)$ is defined analogously to group cohomology, using pro-analytic cochains instead of continuous cochains. This complex was studied in \cite{FX12, BFFanalytic, colmez2016representations}, but the considerations are restricted to low cohomological degrees (i.e.\ zero and one) or $(\varphi_L,\Gamma_L)$-modules attached to $L$-analytic characters.  At the expense of having to extend coefficients to a large (transcendental if $L \neq \Q_p$) field extension $K$ of $L,$ we can describe analytic cohomology by a complex similar to the classical complex considered by Herr. This allows us to even treat $(\varphi_L,\Gamma_L)$-modules over relative Robba rings $\cR_A=A \hat{\otimes}_K \cR_K,$ where $A$ is $K$-affinoid in the sense of Tate. These modules should be thought of as families parametrised by the rigid analytic space $\operatorname{Sp}(A).$ 
In our case the group $\Gamma_L \cong o_L^{\times}$ contains an open subgroup $U \cong o_L$ of finite index and 
by the work of Schneider and Teitelbaum there exists a rigid $L$-analytic variety $\mathfrak{X}_{U},$ whose points parametrise the locally $L$-analytic characters on $U,$ such that the global sections of its structure sheaf are given by the algebra of locally $L$-analytic $L$-valued distributions $D(U,L)$ via the Fourier isomorphism. They further show that over $K$ containing a certain period $\Omega_L \in \CC_p$ the variety $\mathfrak{X}_{U}$ is isomorphic to the open unit disc.
This provides us with an  isomorphism $D(U,K) \cong \cR_K^{[0,1)}$ and we denote by $Z$ the preimage of a coordinate $T$ under this isomorphism. By transport of structure it is evident that $$D(U,K) \xrightarrow{Z} D(U,K) \to K \to 0$$ is a projective resolution of $K$ and hence the complex $M \xrightarrow{Z} M$ represents $\mathbf{RHom}_{D(U,K)}(K,M).$ Kohlhaase's analytic cohomology theory (cf. \cite{Kohlhaase}) provides us with a comparison isomorphism to $C^{\bullet}_\text{pro-an}(U,M)$ (cf. Corollary \ref{cor:comparison}). Note that Kohlhaase works over a spherically complete field and additionally endows the cohomology groups with topologies. The relevant parts of his construction also work in our case, but the above comparison is merely of algebraic nature. This leads us to studying the analytic Herr complex given by the total complex
$$C_{\varphi_L,Z}(M):=Tot\left(\begin{tikzcd} M \arrow{d}{Z} \arrow{r}{\varphi_L-1}& M\arrow{d}{-Z} \\	 M \arrow{r}{\varphi_L-1} &M\end{tikzcd}\right).$$ 
Our work is inspired by the work of Kedlaya, Pottharst and Xiao in \cite{KPX} and specialises to their situation in the case $L=\Q_p,$ where one typically (e.g. for $p\neq 2$) takes $Z = \gamma-1$ for a topological generator $\gamma \in  \Gamma_{\QQ_p}.$ The main difference is having to pass to a large extension, to be able to describe the distribution algebra in terms of a single operator $Z$ and the implicit nature thereof. 
	The starting point is the study of the $\Gamma_L$-action on $M^{\Psi=0}.$ By transport of structure via the isomorphism $\cR_K^{[0,1)} \cong D(U,K)$ we can define the group Robba ring $\cR_K(U)$ (and similarly $\cR_K(\Gamma_L)$). We denote by $\eta(1,T)$ the power series corresponding to the dirac distribution $\delta_1$ attached to $1 \in o_L \cong U$ under this isomorphism. Comparing the action of $Z$ and the action of $T$ allows us to show the following theorem:
\begin{introtheorem} [Theorem \ref{thm:kerpsiprojective}]
	Let $M$ be an $L$-analytic $(\varphi_L,\Gamma_L)$-module over $\cR_A$ admitting a model\footnote{A model is a $(\varphi_L,\Gamma_L)$-module $M^{[r_0,1)}$ over $\cR_A^{[r_0,1)}$ such that $M = \cR_A \otimes_{\cR_A^{[r_0,1)}}M^{[r_0,1)}.$} over $[r_0,1),$ then there exists $r_1 \geq r_0$ such that for any $r\geq r_1$ the $\Gamma_L$-action on $(M^{[r,1)})^{\psi=0}$ extends to an action of $\cR_A^{[r,1)}(\Gamma_L)$ with respect to which $(M^{[r,1)})^{\psi=0}$ is finite projective of rank $\operatorname{rank}_{\cR_A}(M).$ 
\end{introtheorem}
In the cyclotomic case this is \cite[Theorem 3.1.1]{KPX} and this theorem was proven by Schneider and Venjakob in \cite{SchneiderVenjakobRegulator} for free modules over $\cR_K.$
This result shows that the variable $Z$ has properties analogous to the operator $\gamma-1$ studied in the classical case. As an immediate corollary we obtain a comparison isomorphism between the $(\Psi,Z)$-cohomology (defined using the left-inverse $\Psi$ of $\varphi_L$) and $(\varphi_L,Z)$-cohomology.
\begin{cor*}
	Let $M$ be an $L$-analytic $(\varphi_L,\Gamma_L)$-module over $\cR_A.$ The morphism of complexes 
	$$	\begin{tikzcd}
		C_{\varphi_L,Z}(M)  \arrow[d]: &M \arrow[d,"{\operatorname{id}}"] \arrow[r ]      & M \oplus M \arrow[d, "-\Psi\oplus\operatorname{id}"] \arrow[r] & M \arrow[d, "-\Psi"]   \\
		C_{\Psi,Z}(M): 	&M \arrow[r] & M \oplus M\arrow[r]                                          & M                                               
	\end{tikzcd}$$
	is a quasi-isomorphism.
\end{cor*}

Next we study finiteness and base change properties of $C_{\varphi_L,Z}(M).$ Here we deviate from the approach of \cite{KPX}, who deduce finiteness as a consequence of finiteness of Iwasawa cohomology, and instead opt for an approach using methods of \cite{kedlaya2016finiteness}. The finiteness of $(\varphi_L,Z)$-cohomology can be deduced systematically from the general finiteness statements in \cite{kedlaya2016finiteness}.  We denote by $\mathbf{D}^{[0,2]}_{perf}(A)$ the full subcategory of the derived category consisting of complexes, which are quasi-isomorphic to a complex of finitely generated projective modules concentrated in degrees $[0,2].$
\begin{introtheorem}[Theorem \ref{thm:perfect}]
	Let $A,B$ be $K$-affinoid and let $M$ be an $L$-analytic $(\varphi_L,\Gamma_L)$-module over $\cR_A.$ Let $f\colon A  \to B$ be a morphism of $K$-affinoid algebras. Then:
	\begin{enumerate}[1.)]
		\item $C_{\varphi_L,Z}(M) \in \mathbf{D}^{[0,2]}_{\text{perf}}(A).$
		\item The natural morphism $C_{\varphi_L,Z}(M) \otimes_{A}^\mathbb{L} B \to C_{\varphi_L,Z}(M \hat{\otimes}_A B)$ is a quasi-isomorphism.
	\end{enumerate}
	In particular the cohomology groups $H^i_{\varphi_L,Z}(M)$ are finite $A$-modules for every $i.$
\end{introtheorem}
In the last section we explicitly describe a comparison between $H^1_{\varphi,Z}(M)^{\Gamma_L}$ and the group of analytic extensions $\operatorname{Ext}^1_{an}(\cR_A,M)$ and explain the comparison to pro-analytic cohomology using Kohlhaase's theory.
The finiteness of analytic cohomology is used in our subsequent article \cite{RustIwasawa} to study Iwasawa cohomology of analytic $(\varphi_L,\Gamma_L)$-modules. 
\section*{Acknowledgements}
This article is based on parts of the author's Ph.D.\ thesis (cf. \cite{rusti}) and we thank Otmar Venjakob for his guidance and many helpful discussions. We also thank Jan Kohlhaase for his valuable remarks  on \cite{rusti}. This research was funded by the Deutsche Forschungsgemeinschaft (DFG, German Research Foundation) TRR 326 \textit{Geometry and Arithmetic of Uniformized Structures}, project number 444845124.

	\section{Preliminaries}
\subsection{Robba rings and their $(\varphi_L,\Gamma_L)$-actions.}
Let $\varphi_L(X) \in o_L \llbracket X\rrbracket$ be a Frobenius power series for the uniformiser $\pi_L \in o_L,$  i.e., a series satisfying $$\varphi_L(X)=\pi_LX + \text{terms of higher order}$$ and $$\varphi_L(X)\equiv X^q \mod \pi_L.$$
There is a unique Lubin-Tate group law $F_{LT}(X,Y)$ and a unique injective homomorphism of rings $$o_L \to \operatorname{End}(LT),$$ mapping $a\in o_L$ to a power series $[a](X)$ such that $[\pi_L](X)=\varphi_L(X).$ Denote by $L_n$ the extension of $L$ that arises by adjoining all $\pi_L^n$-torsion points of the $LT$-group. The set $LT[\pi_L^n]$ of $\pi_L^n$-torsion points carries a natural $o_L$-module structure which respect to which it is a free $o_L/\pi_L^no_L$-module of rank $1$. One can show that $L_n$ is a finite Galois extension of $L$ with Galois group isomorphic to $\operatorname{End}(LT[\pi_L^n]) \cong (o_L/\pi_L^no_L)^\times$ and by passing to the limit one obtains a continuous character $$\chi_{LT}: G_L \to o_L^\times$$ inducing for $L_\infty = \bigcup_{n\geq1} L_n$ an isomorphism $$\Gamma_L:=\Gal(L_\infty/L) \to o_L^\times.$$
We endow $o_L\llbracket X \rrbracket$ with an action of $\Gamma_L$ via $\gamma(f(X))= f( [\chi_{LT}(\gamma)](X))$ and similarly for $\varphi_L.$ We implicitly fix a norm on $\C_p$ extending the norm on $L$ and endow each subfield with said norm. We frequently deal with power series or Laurent series converging
on some annulus $[r,s].$ We will express radii in relative terms in order to avoid having to fix a normalisation.
\begin{defn}
	Let $K \subset \CC_p$ be a complete field.
	We denote by $\mathcal{R}_K^{[r,s]}$ the ring of Laurent series (resp. power series if $r=0$) with coefficients in $K$ that converge on the annulus $r\leq \lvert x\rvert \leq s$ for $r,s \in p^{\QQ}$ and $x \in \CC_p.$
	It is a Banach algebra with respect to the norm $\lvert \cdot  \vert_{[r,s]} := \max (\lvert \cdot \rvert_r, \lvert \cdot \rvert_s).$ We further define the Fréchet algebra $\mathcal{R}_K^r:=\mathcal{R}_K^{[r,1)} := \varprojlim_{r<s<1} \mathcal{R}_K^{[r,s]}$ and finally the \textbf{Robba ring} $\mathcal{R}_K:= \varinjlim_{0\leq r<1} \mathcal{R}_K^{[r,1)}$ endowed with the LF topology.
\end{defn}
The action of $\Gamma_L$ extends by continuity to $\cR_L^I$ for any interval $I=[r,s]\subset [0,1].$ For details concerning the $\varphi_L$-action we refer to \cite[Section 2.2]{Berger}. To ensure that the action of $\varphi_L$ on $\cR_L^I$ is well-defined
one has to assume that the lower boundary $r$ of $I$ is either $0$ or $r > \abs{u}^q=:r_L$ for some (any by \cite[Proposition 1.3.12]{Schneider2017}) non-trivial $\pi_L$-torsion point $u$ of the Lubin-Tate group \footnote{In the second case the assumptions guarantee $\varphi(T)\in (\cR_L^{I^{1/q}})^{\times}.$}.
When $\varphi_L$ acts on $\cR_L^I$ it changes the radius of convergence and we obtain a morphism 
$$\varphi_L: \cR_L^{I}\to \cR_L^{I^{1/q}} .$$ We implicitly assume that $r,s$ lie in $\abs{\overline{\QQ_p}},$ because in this case the algebra $\cR_L^I$ is affinoid (cf. \cite[Example 1.3.2]{lutkebohmert2016rigid}) and this assumption is no restriction when considering the Robba ring due to cofinality considerations.
We henceforth endow the rings 
$\cR_L^{[r,1)}$ (for $r=0$ or $r>r_L$) and $\cR_L$ with the $(\varphi_L,\Gamma_L)$-actions induced by the actions on $\cR_L^I.$ We also work with relative versions of these rings either defined over some affinoid $A$ over $L$ or more generally some complete field extension $K/L$ contained in $\CC_p.$ Before we describe these relative Robba rings we shall discuss some generalities concerning completed tensor products of Fréchet- and LF-spaces.

\begin{defn}
	Let $X,Y$ be (semi-)normed modules over a normed ring $S.$
	On $X \otimes_SY$ we define the tensor product (semi-)norm $$\lvert v\rvert := \inf_{r} \max \lvert x_i\rvert \lvert y_i\rvert,$$ where $r$ ranges over all representations of $v$ as a sum of elementary tensors $v = \sum_{i} x_i \otimes y_i.$
\end{defn}

\begin{defn} 
	The \textbf{projective completed tensor product} of normed $S$-modules is defined as the completion of the usual tensor product with respect to the tensor product norm. 
	If the topologies on $X$ and $Y$ are defined by a family of semi-norms, we can extend this notion in the obvious way. 
	We write $$X \hat{\otimes}_S Y:=X \hat{\otimes}_{S,\pi} Y$$ for the projective completed tensor product.
\end{defn}
\begin{rem}
	\label{rem:Frechetlim}
	Let $X = \varprojlim_{n \in \NN} X_n$ be a Fréchet space over $K$ with Banach spaces $X_n$ and let $W$ be a normed $K$-vector space. Assume that the transition maps $X_{n+1} \to X_n$ have topologically dense image. Then the canonical map $$X \hat{\otimes}_K W \to \varprojlim_n X_n \hat{\otimes}_K W$$ is a topological isomorphism.
\end{rem}
\begin{proof}
	This is a special case of Lemma 2.1.4 in \cite{Berger}. Note that all involved spaces are Hausdorff because they are metrizable.
\end{proof}
\begin{defn}
	Let $V,W$ be locally convex $K$-vector spaces. The \textbf{inductive tensor product topology} is defined as the finest topology such that the bilinear map $$V\times W \to V\otimes_K W$$ is separately continuous.
	We denote the completion of the usual tensor product with respect to that topology by $V \hat{\otimes}_{K,i}W.$ 
\end{defn}
\begin{rem}
	\label{rem:lfcomplete}
	The inductive and projective tensor products agree for Fréchet spaces. The inductive tensor product and its completed version commute with countable strict locally convex inductive limits of Fréchet spaces.
\end{rem}
\begin{proof}
	For the first statement see \cite[Proposition 17.6]{SchneiderNFA}. By \cite[1.1.30]{emerton2017locally} the inductive tensor product commutes with locally convex inductive limits. Let $V = \varinjlim_n V_n$ and $W=\varinjlim_n W_n$ be strict LF-spaces with Fréchet spaces $V_n$ (resp. $W_n$). We already know  that $V \hat{\otimes}_{K,i} W$  is the completion of $\varinjlim_n (V_n \otimes_{K,i} W_n).$ In the proof of \cite[2.1.7]{Berger} it is shown that for an inductive system $(E_n)_n$ of locally convex vector spaces such that $\varinjlim_{n\in \NN} \widehat{E_n}$ is Hausdorff and complete the natural map $$\varinjlim_n\widehat{E_n} \to \widehat{\varinjlim_nE_n}$$ is an isomorphism. By \cite[Theorem 11.2.4 and Theorem 11.2.5]{PGS} we may apply this result to $E_n = V_n \otimes_K W_n,$ which yields the desired result.
\end{proof}

\begin{rem}
	Let $D \subset V$ (resp. $E \subset W$) be dense subsets of locally convex spaces $V,W.$ Then $D \otimes_K V$ is dense in $E \otimes_K W$ and $E \hat{\otimes}_KW.$
\end{rem}
\begin{proof}
	\cite[Corollary 10.2.10]{PGS} shows that the natural map $D \otimes_K E \to V \otimes_K W$ is a topological embedding. Applying Corollary 10.2.10(v) in loc.cit. to each seminorm defining the topology on $V \otimes_KW$ shows that $D \otimes_K E$ is dense in the $V \otimes_K W$. Because $V \otimes_K W\to V \widehat{\otimes}_K W$ is a topological embedding with dense image the statement follows.
\end{proof}
\begin{defn}
	An \textbf{affinoid algebra} $A$ over a non-archimedean complete field $F$ is an algebra that is isomorphic to $$\mathbb{T}^n / I,$$ where $\mathbb{T}^n$ denotes the $n$-dimensional Tate-algebra $F\langle X_1, \dots, X_n\rangle$ and $I\subset \mathbb{T}^n$ is an ideal. We always endow $A$ with the residue norm obtained from the Gauß-norm on $\mathbb{T}^n.$ By \cite[3.1 Proposition 20]{bosch2014lectures} any two residue norms are equivalent and any ideal in $\mathbb{T}^n$ is closed (cf. \cite[Section 2.3]{bosch2014lectures}).
\end{defn}

\begin{defn}
	Let $F\subset K$ be a complete subfield and $A$ be an affinoid algebra over $F.$ We define the \textbf{relative Robba rings}
	$\mathcal{R}_A^{[r,s]}:= \mathcal{R}_K^{[r,s]} \hat{\otimes}_F A $ and similarly  $\mathcal{R}_A^{[r,1)}$ and $\mathcal{R}_A := \varinjlim_{0\leq r<1} \mathcal{R}_A^{[r,1)}.$ These rings are naturally equipped with topolgies induced by the tensor product norm on $\cR^{[r,s]}_A$. 
\end{defn}
\begin{defn}
	A linear map $T:E\to F$ between locally convex $K$-vector spaces is called \textbf{compactoid} if there exists a zero neighbourhood $U\subset E$ such that $T(U)$ is \textbf{compactoid} in $W$ meaning that for every zero neighbourhood $V\subset F$ there exists a finite set $e_1,\dots,e_n$ such that $$T(U) \subset V+ \sum_{i=1}^n o_K e_i.$$
\end{defn} 
The following is a technical subtlety and does not follow from \ref{rem:lfcomplete} because (non-strict) LF-spaces are not automatically complete. 
\begin{lem}
	\label{lem:Banachdirectlimit}
	Let $E = \varinjlim_n E_n$ be an LB-space with $L$-Banach spaces $E_n$ and compactoid steps. Let $W$ be an $L$-Banach space then
	the natural map $$\varinjlim_n( E_n\hat{\otimes}_{L,\pi}W) \to E\hat{\otimes}_{L,\pi}W$$ is an isomorphism. In particular $$E\hat{\otimes}_{L,\pi}W=E\hat{\otimes}_{L,i}W.$$
\end{lem}
\begin{proof}
	By \cite[11.3.5]{PGS} $E$ is complete reflexive and its strong dual $E':=E'_b$ is Fréchet. Furthermore as an inductive limit of bornological spaces $E$ is bornological by \cite[Example 2) after 6.13]{SchneiderNFA}. By \cite[18.8]{SchneiderNFA} together with reflexivity we have $$E \hat{\otimes}_{L,\pi}W = E'' \hat{\otimes}_{L,\pi}W = \mathcal{L}_b(E',W),$$ where $\mathcal{L}_b(-,-)$ denotes the space of continuous linear maps endowed with the strong topology. Furthermore by \cite[Proposition 1.5]{schneiderteitlbaumlocallyanalytic} $$\varinjlim_{n} E_n\hat{\otimes}_{L,\pi} W= \mathcal{L}_b(E',W).$$ Combining the above and unwinding the definitions of the involved maps yields the desired claim. 
\end{proof}
\begin{lem}
	\label{lem:affbasechange}
	Let $F \subset K$ be a complete subfield and $A$ be an affinoid algebra over $F$. Then the natural map induces isomorphisms
	$$\cR_K^I \hat{\otimes}_F A \cong \cR_K^I \hat{\otimes}_K (K \hat{\otimes}_F A)$$
	
	and 
	$$\cR_K \hat{\otimes}_{F,i} A\cong \cR_K \hat{\otimes}_{K,i} (K \hat{\otimes}_F A). $$
\end{lem}
\begin{proof}
	The embedding $F \subset K$ is by construction isometric and thus contracting. Applying \cite[2.1.7 Proposition 7]{BGR} we obtain
	$$\cR_K^I \hat{\otimes}_F A = (\cR_K^I \hat{\otimes}_K K)\hat{\otimes}_F A \cong \cR_K^I \hat{\otimes}_K (K\hat{\otimes}_F A).$$
	The second part follows by taking limits.
\end{proof}
Lemma \ref{lem:affbasechange} allows us to restrict ourselves to the case $F=K$ since the base change $K \hat{\otimes}_F A$ of an affinoid algebra over $F$ is an affinoid algebra over $K$ (cf. \cite[6.1.1. Corollary 9]{BGR}).
\begin{rem}
	\label{rem:robbacomplete} 
	The relative Robba ring $\cR_A$ is complete. In particular $\cR_A = A\hat{\otimes}_{L,i} \cR_L.$ Furthermore $\cR_A = A\hat{\otimes}_{L,\pi} \cR_L.$
\end{rem}
\begin{proof}
	Recall from the proof of \cite[2.1.6]{Berger} that $\cR_L$ admits a decomposition of the form $\cR_L = \cR_L^+ \oplus E$ with an $LB$-space $E= \varinjlim_n E_n$ with compactoid steps. For such spaces it is known that their inductive limit is complete by \cite[11.3.5]{PGS}. We obtain a corresponding decomposition $\cR_A= \cR_A^+ \oplus \varinjlim_n A \hat{\otimes}_{L,i} E_n$ with $\cR_A^+ = A \hat{\otimes}_{L,i}\cR_L^+$ and hence $\cR_A^+ = A \hat{\otimes}_{L,\pi}\cR_L^+$ by \ref{rem:lfcomplete}. The other summand is treated by the preceding Lemma \ref{lem:Banachdirectlimit}.
	
\end{proof}


We have two reasonable choices for the $(\varphi_L,\Gamma_L)$-actions on $\cR_K$ for a complete field $K\subset \CC_p.$ One possibility is the linear extension of the $(\varphi_L,\Gamma_L)$ action from $\cR_L$ using $\cR_K =K\hat{\otimes}_L \cR_L.$ If on the other hand $K$ is invariant under the $G_L$ action on $\CC_p,$ we can take the semi-linear $G_L$-action, which factors over $\Gamma_L$ if $K\subset \widehat{L_\infty}.$ Unless stated otherwise we consider only the former action. This action also makes sense for more general coefficients. Another reason for studying the linear action rather than the semi-linear actions is that we would like to work with $L$-analytic actions. The semi-linear action on say $\widehat{L}_\infty$ will never be $L$-analytic by \cite[Corollaire 4.3]{bergercolmez2016theoriedesen}.
\subsection{$p$-adic Fourier theory and $D(G,L)$ actions.} We give a survey of $p$-adic distributions, that play a crucial role in the study of $L$-analytic $(\varphi,\Gamma_L)$-modules. In the case $G=\mathbb{Z}_p$ and $L=\mathbb{Q}_p$ a theorem of Amice asserts that $D(G,\QQ_p)$ is isomorphic to the holomorphic functions on the open unit disc in the variable $\mathfrak{Z}=\delta_1-1.$ This facilitates the study of the $\Gamma_{\QQ_p}$-action in the classical theory. If $L \neq \QQ_p$ a similar result can only be achieved after passing to a large extension $K$ of $L$ that contains a certain (transcendental) period $\Omega_L \in \CC_p.$ In the notation of \cite{schneider2001p} the period can be taken to be $\Omega_L:= \Omega_t$ for some basis $t$ of the Tate module of the dual of the Lubin-Tate group.
\begin{defn}
	Let $G$ be a compact $L$-analytic group. We denote by $D_{\QQ_p}(G,K)$ the algebra of $\QQ_p$-analytic distributions with values in $K,$ which is the strong dual of the space $C^{\QQ_p\text{-an}}(G,K)$ of locally $\QQ_p$-analytic functions on $G$ with values in $K$ with multiplication given by convolution. We denote by $\delta_g$ the Dirac distribution associated to $g,$ by which we mean the evaluation map $\delta_g\colon f \mapsto f(g).$ We denote by $D(G,K)$ the quotient of $D_{\QQ_p}(G,K)$  corresponding to the dual of the subspace $C^{an}(G,K):=C^{L-an}(G,K)$ of $C^{\QQ_p\text{-an}}(G,K)$ consisting of locally $L$-analytic functions.
\end{defn}
For a detailed description of the topology on $C^{an}(G,K)$ we refer the reader to \cite[Chapters 10 and 12]{Schneider2017}. In \cite[Section 2]{schneider2001p} Schneider and Teitelbaum construct the character variety $\widehat{G}$ over $L$ whose $K$-points parametrise locally $L$-analytic $K$-valued characters of the additive group $G=o_L.$ For a complete field extension $K$ of $L$ we denote by $\mathcal{O}_K(\widehat{G})$ the global sections of the base change of said variety to $K.$ We denote by $\log_{LT}(T) \in \cR_L^{[0,1)}$ the logarithm of the formal Lubin-Tate group.
\begin{thm}
	\label{thm:Fourier}
	Let $G=o_L$ viewed as an $L$-analytic group in the natural sense and let $L \subset K \subset \CC_p$ be a complete intermediate field. Then the Fourier transform\footnote{ Defined on p. 452 in loc.\ cit..} induces an isomorphism of $K$-Fréchet algebras
	$$D(G,K) \to \mathcal{O}_K(\widehat{G}).$$
	If $K$ contains a period $\Omega_L$  of the Lubin-Tate group, then $\widehat{G}$ and the open unit disc $\mathbb{B}$ are isomorphic over $K$ and by combining the above with the Fourier isomorphism we obtain an isomorphism $$D(G,K) \xrightarrow{\cong} \mathcal{O}_K(\mathbb{B}).$$
	By choosing a coordinate $T$ on $\mathbb{B}$ it can be described explicitly by mapping a dirac distribution $\delta_a$ to the power series $$\eta(a,T)=\exp(a\Omega_L\log_{LT}(T))\in o_K\llbracket T \rrbracket.$$
\end{thm}
\begin{proof}
	This follows by combining Corollary 3.7 and Theorem 2.3. in \cite{schneider2001p}.
\end{proof}
We remark that the above isomorphism is in fact an isomorphism of group varieties when $\mathbb{B}$ is identified with the group object attached to the formal Lubin-Tate group.
Henceforth we assume that $K$ contains a period $\Omega_L$ as above.
\begin{rem}
	\label{bem:directsumdistribution}
	If $H\subset G$ is an open normal subgroup then the decomposition $G= \bigcup_{g \in G/H} gH$ induces $D(G,K) \cong \bigoplus_{g \in G/H} \delta_gD(H,K) \cong \ZZ[G] \otimes_{\ZZ[H]}D(H,K)$ algebraically and topologically. 
\end{rem}
We are mostly interested in the case where $G$ is abelian and contains an open subgroup isomorphic to $o_L$ (like $\Gamma_L$).
For technical purposes it is important that $C^{an}(G,K)$ can be written as a compactoid inductive limit (cf. \cite{rusti} for details). It then follows from \cite[11.3.5]{PGS} that $C^{an}(G,K)$ and $D(G,K)$ are reflexive, strictly of countable type and satisfy Hahn-Banach. As a consequence one obtains the following:

\begin{rem} \label{rem:hahnbanach}
	Mapping $g$ to $\delta_g$ induces an injective group homomorphism $G \to D(G,K)^\times$ and an injection $K[G]\hookrightarrow D(G,K).$ This injection has dense image.
\end{rem}

\begin{defn}
	Let $G=o_L.$ After fixing a coordinate $T$ on $\mathbb{B}$ we denote by $Z\in D(o_L,K)$ the preimage of $T\in \mathcal{O}_K(\mathbb{B})$ with respect to the isomorphism from Theorem \ref{thm:Fourier}.
\end{defn}
In the classical theory (assuming $p \neq 2$) we can choose explicitly $Z=\gamma-1$ with a topological generator of  $\Gamma_{\QQ_p}.$ In our situation this variable $Z$ serves a similar purpose but is more elusive in its description. The main difficulty is reversing $\eta$ since there is no (evident) connection between the exponential and the Lubin-Tate logarithm unless $L=\QQ_p$ and $LT=\mathbb{G}_m.$

\begin{rem}\label{rem:augmentation} Let $\operatorname{Aug}\colon D(o_L,K)\to K$ be the augmentation map induced by mapping each Dirac distribution to $1$ and denote by $\operatorname{ev_0}$ the map that evaluates a power series at $T=0.$ Then the following diagram is commutative $$\begin{tikzcd}
		{D(o_L,K)} \arrow[r, "\operatorname{Aug}"] \arrow[d, "{\cong}"]& K \arrow[d, no head,equal ] \\
		\mathcal{O}_K(\mathbb{B}) \arrow[r, "{\operatorname{ev_0}}"]                        & K              
	\end{tikzcd}.$$
	In particular 
	$$\operatorname{ker}(\operatorname{Aug}) = \overline{\operatorname{span}(\delta_a-1, a \in G)} = ZD(o_L,K).$$
	
\end{rem}
\begin{proof}
	The vertical arrows are topological isomorphisms for the Fréchet-topology on the left side (resp. the valuation topology on the right side) and the Dirac distributions span a dense subspace of $D(o_L,K).$ We may therefore check the commutativity of said diagram on the Dirac distributions, where we have $$\operatorname{Aug}(\delta_a)=1= \exp(a\Omega_{L}\log_{LT}(T))_{\mid T=0} = \operatorname{ev_0}(\eta(a,T)).$$ The value at $T=0$ is independent of the choice of coordinate.  For the second statement we first remark that both maps are surjective and their kernels are mapped isomorphically to one another. Evidently $\operatorname{ker}(\operatorname{ev}_0) = T\mathcal{O}_K(\mathbb{B})\cong ZD(o_L,K).$ Due to continuity the inclusion \linebreak $\operatorname{ker}(\operatorname{Aug}) \supseteq \overline{\operatorname{span}(\delta_a-1)}$ is clear. For the other inclusion consider the decomposition $K[o_L]\cong K\delta_0 \oplus \operatorname{span}(\delta_a-1, a \in o_L)$ with the augmentation map restricted to the first factor mapping $\lambda \delta_0$ to $\lambda$. The left factor is clearly complete and passing to completions shows the desired result.
\end{proof}
\begin{rem}\label{rem:gammacomm} Let $a\in o_L$ and denote by $a^*$ the map induced from the multiplication-by-$a$-map $a:o_L \to o_L.$ Then the following diagram is commutative $$\begin{tikzcd}
		{D(o_L,K)} \arrow[r, "a^*"] \arrow[d, no head, equal ] & {D(o_L,K)} \arrow[d, no head,equal ] \\
		\mathcal{O}_K(\mathbb{B}) \arrow[r, "{[a]}"]                        & \mathcal{O}_K(\mathbb{B})                
	\end{tikzcd},$$
	where the vertical arrows arise from the Fourier-isomorphism. 
\end{rem}
\begin{proof}
	Using $\log_{LT}([a](T))=a\log_{LT}(T)= [a]\log_{LT}(T)$ we obtain $$a^*(\delta_b)=\exp(\Omega_L ab\log_{LT}(T))=[a](\exp(\Omega_Lb\log_{LT}(T))) = [a](\delta_b),$$ proving the result for Dirac distributions. The general statement follows from continuity considerations. 
\end{proof}
\begin{lem}
	The kernel of the natural map $\operatorname{proj}\colon D(o_L,K)\to K[o_L/\pi_L^no_L]$ is generated by $Z_n:=[\pi_L]^n(Z) \in D(o_L,K).$
\end{lem}
\begin{proof}	Since $Z$ lies in the closure of the augmentation ideal generated by the $\delta_g-1$ with $g \in o_L$ using \ref{rem:gammacomm} we see that $Z_n$ lies in the closure of the augmentation ideal of $D(\pi_L^no_L,K)\subset D(o_L,K).$ We conclude $Z_n\subset \operatorname{ker}(\operatorname{proj}).$ By transporting the structure to $\mathcal{O}_K(\mathbb{B})$ we see that  $\mathcal{O}_K(\mathbb{B})/\varphi_L^n(T)$ is free of rank $q^n$ over $K$ by counting the number of zeros of $\varphi_L^n,$ i.e., the number of $\pi_L^n$-torsion points of the LT group. We conclude that the surjective map $\operatorname{proj}$ has to be injective modulo $[\pi_L]^n(Z).$
\end{proof} So far we only needed to choose a variable for $D(o_L,K).$ Analogously we can choose variables for any subgroup $\pi_L^no_L$ since they are isomorphic to $o_L.$ The following corollary shows that the ideal generated by such a variable is independent of any such choices. We state the result on the level of $\Gamma_L$ since we use it in this particular context later on. Let $\Gamma_0:= \Gamma_L$ and set $\Gamma_n := \chi_{LT}^{-1}(1+\pi_L^no_L)$ for $n \geq 1.$
\begin{cor}
	\label{kor:modZgruppenring}
	Let $n_0$ be minimal with the property that $\chi_{LT}(\Gamma_{n_0})=1+\pi_L^{n_0}o_L$ is isomorphic to $\pi_L^{n_0}o_L$ via $\log_p.$ 
	For $n \geq n_0$ denote by $Z_n$ the preimage of $T$ under the sequence of isomorphisms $$D(\Gamma_n,K)\cong D(o_L,K)\cong \mathcal{O}_K(\mathbb{B})$$ viewed as an element of $D(\Gamma_{n_0},K).$ Then there is a canonical $\Gamma_{n_0}/\Gamma_n$ equivariant isomorphism
	$D(\Gamma_{n_0},K)/Z_n \cong K[\Gamma_{n_0}/\Gamma_n]$ induced by mapping $\delta_g$ to $g.$
\end{cor}
\begin{proof}
	This follows from transport of structure along the isomorphisms $\chi_{LT}: \Gamma_L \to o_L^\times$ and $\log_p: 1+\pi_L^no_L\to \pi_L^no_L \cong o_L.$ The isomorphism is canonical in the sense that neither its definition nor its kernel depends on the particular choice of $Z_n.$
\end{proof}
Note that in the above construction $Z_n$ is just the preimage of $Z$ under the map induced by the chart $\Gamma_n \cong o_L.$ Choosing our charts in a specific way leads us to the following compatible choice of variables, which we fix for the remainder of the article. 
\begin{defn} \label{def:variable}
	Consider for $n\ge n_0$ the system of commutative diagrams 
	$$	\begin{tikzcd}
		\Gamma_n \arrow[r, "\chi_{LT}"]                 & 1+\pi_L^n \arrow[r, "\log_p"]                 & \pi_L^no_L                 \\
		\Gamma_m \arrow[u, hook] \arrow[r, "\chi_{LT}"] & 1+\pi_L^m \arrow[u, hook] \arrow[r, "\log_p"] & \pi_L^mo_L \arrow[u, hook]
	\end{tikzcd}
	$$
	We fix as before the variable $Z_{m} \in D(\Gamma_m,K)$ for every $m\geq n_0$ (induced by the charts $\Gamma_m \cong \pi_L^mo_L\cong o_L$ above and a fixed choice of $Z \in D(o_L,K)$).  Since $\pi_L^{n}o_L = \pi_L^{m-n} (\pi_L^no_L)$ we obtain the relationship \begin{equation}\label{eq:Znbeziehung} Z_m = \varphi_L^{m-n}(Z_n)\end{equation} for every $m\geq n.$ Here $Z_m$ is viewed as an element of $D(\Gamma_n,K)$ via the map induced by the canonical inclusion.
\end{defn}

\subsection{$L$-Analyticity}
In this section we discuss the question of $L$-analyticity in families. 
Abstractly $L$-analyticity can be defined for any $K$-Banachspace with a continuous $K$-linear $\Gamma_L$-action where $L\subset K$ is a complete field extension. In our application families of $(\varphi_L,\Gamma_L)$-modules will be LF-spaces over $K$ with an $\cR_A$-semi-linear (hence a fortiori $K$-linear) $\Gamma_L$-action. 
\begin{defn}	Let $X$ be a $d$-dimensional $L$-analytic manifold and let $F$ be a complete subfield of $\CC_p.$ 
	\begin{itemize}
		\item Let $V$ be a $F$-Banach space. A map $f:X \to V$ is called \textbf{locally $L$-analytic} if for every $x \in X$ there exists an open neighbourhood $U$ homeomorphic to $B_{\varepsilon}(0)^d$ and $(v_\mathbf{n})\in V^{\NN_0^d}$ such that $$\lim_{\lvert \mathbf{n}\vert \to \infty}\varepsilon^{\abs{\mathbf{n}}} \lvert \lvert v_\mathbf{n}\rvert \rvert=0$$ and
		$$f(x)=\sum_{\mathbf{n}\in\NN_0^d}v_\mathbf{n}(x_1,\dots,x_d)^\mathbf{n}$$
		for every $x \in U.$ Where $x_i$ are local coordinates of $x$ and $(x_1,\dots,x_d)^\mathbf{n} = \prod_i x_i^{n_i}.$ 
		\item For a $F$-Fréchet space $\varprojlim_j V_j$ with Banachspaces $V_j$ a map $f:X \to \varprojlim_j V_j$ is called \textbf{pro-$L$-analytic} if each induced map $X \to V_j$ is locally $L$-analytic. 
		\item For an LF-space $\varinjlim_i\varprojlim_j V_{i,j}$ with Banachspaces $V_{i,j}$ a map $f: X \to \varinjlim_i \varprojlim_j V_{i,j}$ is called \textbf{pro-$L$-analytic} if it factors over some $\varprojlim_jV_{i,j}$ and the induced map is pro-$L$-analytic in the Fréchet-sense.
	\end{itemize}
\end{defn}
\begin{defn}	Let $V=\varinjlim_i\varprojlim_j V_{i,j}$ be a $F$-LF-space and let $G$ be a $p$-adic Lie group over $L$ acting on $V.$
	The action is called \textbf{$L$-analytic} if for each $v \in V$ the orbit map $g \mapsto gv$ is pro-$L$-analytic.
\end{defn}
In order to treat tensor products of $(\varphi_L,\Gamma_L)$-modules we introduce the following auxiliary definition. This property will be satisfied in our applications.
\begin{defn}	Let $V=\varinjlim_{i}\varprojlim_j V_{i,j}$ be a $K$-LF-space with an $L$-analytic action of an $L$-analytic group $G.$ We say that $V$ \textbf{admits a model} if there exists a $m_0$ such that the orbit map of each $v \in \varprojlim_jV_{m,j}$ already factors over $\varprojlim V_{m,j}$ whenever $m \geq m_0.$
\end{defn}

\begin{lem}
	\label{def:anal}
	Let $F/L$ be a complete field extension and $V$ be a $K$-Banach space with a continuous $F$-linear $\Gamma_L$-action. The $\Gamma_L$-action on $V$ is locally $L$-analytic if and only if the following conditions are satisfied:
	\begin{itemize}
		\item There exists $m \geq2$ such that $\lvert \lvert \gamma-1 \rvert \rvert_V < p^{-1/(p-1)}$ for any $\gamma \in \Gamma_m.$
		\item The derived action $\operatorname{Lie}(\Gamma_L) \to \operatorname{End}_F(V)$ is $L$-linear.
	\end{itemize}
\end{lem}
\begin{proof}
	The first condition in \ref{def:anal} asserts that the $\Gamma_L$-action is locally $\QQ_p$-analytic by \cite[Lemma 2.3.1]{Berger}. The second condition then implies $L$-analyticity (cf. \cite[Folgerung 2.2.10]{Max}).
\end{proof}
\begin{lem} \label{lem:analytictensor} Let $V= \varinjlim_{i \in \N} V_i,W= \varinjlim_{i \in \N} W_i$ be two $F$-LF-spaces with pro $L$-analytic actions (with Fréchet spaces $V_i,$ $W_i$), that both admit models. Then the tensor product action on $E:=\varinjlim_{i}V_i\hat{\otimes}_{F,i}W_i$ is pro $L$-analytic.
\end{lem}

\begin{proof} Write $V= \varinjlim_i \varprojlim_j V_{ij}$ (resp. $W=\varinjlim_m\varprojlim_nW_{mn}$) with Banach spaces $V_{ij}$ (resp. $W_{mn}$) with dense transition maps \footnote{Given a Fréchet space $E$ one can always write $E$ as a projective limit of Banach spaces $E_n$ with dense transition maps by taking an increasing sequence of semi-norms $q_n$ inducing the topology on $E$ and taking $E_n$ to be the Hausdorff completion of $E$ with respect to $q_n.$ }. For $E$ we have $E = \varinjlim_{m,i}\varprojlim_{j,n} V_{ij}\hat{\otimes}W_{mn}$ thus
	using that $V$ and $W$ admit models we see that $E$ also admits a model. 
	We may therefore without loss of generality assume that $V,W$ are Banach spaces. In this case the $\QQ_p$-analyticity of the action follows from \cite[3.3.11.]{Feaux}. By \cite[3.3.12]{Feaux} the derived action is given by $D(\rho_V) \otimes \id_W + \id_V \otimes D(\rho_W),$ where $D(\rho_X)$ denotes the derived action on $X\in \{V,W\}.$ This action is $L$-linear as a sum of $L$-linear maps.
\end{proof}


For certain technical arguments we require an analogue (on the distribution side) of $r$-convergent power series.
\begin{defn}
	\label{def:rkonvergenteDist}
	Let $G \cong o_L,$ let $g_1,\dots,g_d$ be a $\ZZ_p$-basis of $G$ and let $r \in [\abs{p},1).$ We define $D_{\QQ_p,r}(G,K)$ to be the completion of $D_{\QQ_p}(G,K)$ with respect to the norm $$\abs{ \sum_{\mathbf{k} \in \NN_0^d} a_\mathbf{k} (\delta_{g_i}-1)_i^\mathbf{k}}_{{D_{\QQ_p,r}(G,K)}}  = \sup_\mathbf{k} \lvert a_\mathbf{k}\rvert r^{\lvert \mathbf{k}\rvert}$$ with the usual conventions for multi-indices, i.e., $((\delta_{g_i}-1)_i)^{(k_1,\dots,k_d)} = \prod_{i=1}^d(\delta_{g_i}-1)^{k_i}$ and $\lvert \mathbf{k}\rvert = \sum_{i=1}^d k_i.$ We further define $D_{r}(G,K)$ as the completion of $D(G,K)$ with respect to the quotient norm with respect to the natural projection $D_{\QQ_p}(G,K) \to D(G,K).$
\end{defn}
\begin{rem} Choose an ordered $\ZZ_p$-basis $h_1,\dots,h_d$ of $o_L$ and let $b_i:= \delta_{h_i}-1 \in D_{\QQ_p}(o_L,K)$ and $\mathbf{b}= (b_1,\dots,b_d).$
	Any element $\lambda \in D_{\QQ_p}(o_L,K)$ admits a unique convergent expansion 
	$$\lambda = \sum_{\mathbf{k} \in \NN_0^d}a_\mathbf{k} \mathbf{b}^\mathbf{k}$$
	such that $\abs{\lambda}_{D_{\QQ_p,r}(G,K)}=\sup \abs{a_k}r^{\abs{\mathbf{k}}}$ is bounded for any $0<r<1.$ The norms $\abs{-}_{D_{\QQ_p,r}(G,K)}$ for $r \in p^{\QQ}, \lvert p\rvert<r<1$ are sub-multiplicative and independent of the choice of ordered basis. 
\end{rem}
\begin{proof}
	See \cite[4.2 and discussion after 4.10]{schneider2003algebras}.	
\end{proof}
\begin{lem}
	\label{lem:extendscalaraction}
	Let $I =[r,s] \subset [0,1)$ and let $\mathfrak{r}>s\geq\lvert p\rvert$ then the homomorphism induced by the composition of the natural projection and the LT-isomorphism
	$$D_{\QQ_p}(o_L,K) \to D(o_L,K) \to \mathcal{O}_K(\mathbb{B})$$
	extends to a continuous homomorphism
	$$D_{\QQ_p,\mathfrak{r}}(o_L,K) \to D_{\mathfrak{r}}(o_L,K) \to \mathcal{O}_K(\mathbb{B}^{[0,s]}) \subset \mathcal{O}_K(\mathbb{B}^I)$$ of operator norm $1.$
\end{lem}
\begin{proof}
	
	Consider the composite map $$D_{\QQ_p}(o_L,K) \to D(o_L,K) \to \mathcal{O}_K(\mathbb{B}) \to \mathcal{O}_K(\mathbb{B}^{[0,s]}),$$ where the last arrow is the canonical inclusion. Since the target is complete (with respect to the $s$-Gauß norm) it suffices to show that the map is continuous with respect to the $\mathfrak{r}$-norm on the left-hand side. 	The series $\eta(a,T)-1 = a\Omega T+\dots$ has no constant term and $\abs{a\Omega}<1$ for every $a \in o_L.$ In particular $\abs{\eta(a,T)-1}_{I}<s$ by assumption. Choose as before a $\ZZ_p$-basis $h_1,\dots,h_d$ of $o_L$ and let $\mathbf{e} = (\eta(h_1,T)-1,\dots,\eta(h_d,T)-1).$
	Then $\lambda =\sum_{\mathbf{k} \in \NN_0^d}a_{\mathbf{k}} \mathbf{b}^{\mathbf{k}}$ is mapped to
	$\sum_{\mathbf{k} \in \NN_0^d}a_\mathbf{k} \mathbf{e}^\mathbf{k}$ and 
	
	$${\abs*{\sum_{\mathbf{k} \in \NN_0^d} a_\mathbf{k} \mathbf{e}^{\mathbf{k}}}}_I \leq \sup \abs{a_\mathbf{k}}s^{\abs{\mathbf{k}}} \leq \abs{\lambda}_{D_{\QQ_p,\mathfrak{r}}(o_L,K)}.$$
	
	This shows that the operator norm is bounded by $1.$ It has to be equal to $1$ because the scalars are mapped to themselves.
\end{proof}

\subsection{$(\varphi_L,\Gamma_L)$-modules over the Robba ring.}
When studying $(\varphi_L,\Gamma_L)$-modules over the Robba ring $\cR_K$ it turns out that they admit a so-called model over some half-open interval $[r,1),$ meaning that it arises as a base extension from a module over $\cR_K^{[r,1)}.$ When working with families of such modules, we enforce the existence of such a model, which in turn allows us to view the modules (more precisely their models) as vector bundles on $\operatorname{Sp}(A) \times_K \mathbb{B}^{[r,1)},$ where $\operatorname{Sp}(A)$ denotes the rigid analytic variety attached to $A$.  A suitable frame work to do so is the theory of coadmissible modules in the sense of Schneider and Teitelbaum. 
\begin{defn}
	A commutative $K$-Fréchet algebra $\mathcal{A}$ is called \textbf{Fréchet-Stein algebra} if there is a sequence of continuous algebra seminorms $q_1 \leq \dots \leq  q_n \leq \dots $ that define the Fréchet-topology such that 
	\begin{enumerate}[(i)]
		\item The completion $\mathcal{A}_n$ of $\mathcal{A}/\{a \mid q_n(a)=0\}$ with respect to $q_n$ is Noetherian and
		\item $\mathcal{A}_{n}$ is flat as an $\mathcal{A}_{n+1}$-module for any $n \in \NN.$
	\end{enumerate}
	A \textbf{coherent sheaf} is a family $(M_n)_{n \in \NN}$ of finitely generated $\mathcal{A}_n$-modules endowed with their respective canonical topologies such that $$\mathcal{A}_{n+1} \otimes_{\mathcal{A}_n} M_n \cong M_{n+1}.$$ The \textbf{global sections} of a coherent sheaf are defined to be the Fréchet-module $$\Gamma(M_n):=\varprojlim_{n \in \NN} M_n.$$ An $\mathcal{A}$-module is called \textbf{coadmissible} if it arises as the global sections of a coherent sheaf.
\end{defn}
\begin{lem}
	\label{lem:frechetstein}
	Let $\mathcal{A}$ be a Fréchet-Stein algebra and let $M$ be coadmissible. Then 
	\begin{enumerate}[(i)]
		\item $\mathcal{A}_n$ is flat as an $\mathcal{A}$-module.
		\item The canonical map $M \to M_n$ has dense image and $\mathcal{A}_n \otimes_A M \cong M_n.$
		\item $\varprojlim^{i}_{n\geq1} M_n =0$ for any $i \geq 1.$
		\item Kernels, cokernels, images and coimages of $\mathcal{A}$-linear maps between coadmissible modules are coadmissible.
		\item Finitely generated submodules of coadmissible modules are coadmissible.
		\item Finitely presented modules are coadmissible.
		\item Let $A$ be $K$-affinoid, then $\cR_A^{[r,1)}$ is a Fréchet-Stein algebra for any $r \in [0,1).$
	\end{enumerate}
\end{lem}
\begin{proof}
	For \textit{(i)-(vi)} see section 3 in \cite{schneider2003algebras}. \textit{(vii)} is well known for finite extensions $K/\Q_p$ (cf. \cite[Section 2.1]{KPX}). The same arguments apply in the general case  (cf. \cite[Appendix]{rusti}).
\end{proof}
Using the theory of coadmissible modules we can deduce the following useful result.
\begin{lem}
	\label{lem:finitecomplete} Let $\mathfrak{m} \in \operatorname{Max}(A)$ be a maximal ideal, then the natural map $$\cR_A \otimes_A A/\mathfrak{m} \to \cR_{A/\mathfrak{m}}$$ is an isomorphism.
\end{lem}
\begin{proof} 
	By a limit argument it suffices to show that  $\cR_A^{[r,1)}\otimes_A A/\mathfrak{m} \to \cR^{[r,1)}_{A/\mathfrak{m}}$ is an isomorphism for every $0<r<1.$ The right-hand side is by definition the Hausdorff completion of the left-hand side. Hence we are done if we can show that ${\cR_A^{[r,1)}\otimes_A A/\mathfrak{m}}= \cR_A^{[r,1)}/\mathfrak{m}\cR_A^{[r,1)}$ is complete. Since $A$ is Noetherian the ideal $\mathfrak{m}\cR_A^{[r,1)}$ is a finitely generated submodule of a coadmissible module hence itself coadmissible. By \cite[Lemma 3.6]{schneider2003algebras} the quotient in question is complete.
\end{proof}
The definition of coherent sheaves so far is too restrictive because we can not change the lower boundary of a given half-open interval $[r,1).$ Recall that a collection of subsets of a topological space is called locally finite if every point admits an open neighbourhood intersecting only finitely many members of the collection.
\begin{rem}
	\label{rem:glueing}
	Let $\{[r_i,s_i], i\in \NN_0\}$ be an admissible cover of $[r,1)$ i.e. a cover by closed intervals that admits a locally finite refinement with $r_i,s_i \in \sqrt{\lvert K\rvert^\times}$. For each $i$ let  $M^{[r_i,s_i]}$ be a finitely generated $\cR_A^{[r_i,s_i]}$-module together with isomorphisms 
	$$\cR_A^{I\cap J} \otimes_{\cR_A^I}M^I\cong \cR_A^{I\cap J} \otimes_{\cR_A^J}M^J$$
	for any pair of intervals with non-empty intersection, satisfying the obvious compatibility conditions. Then for each $s \in [r,1)$ there exists a unique coadmissible $\cR_A^{[s,1)}$-module together with morphisms $M^{[s,1)} \to \cR_A^{I\cap [s,1)}\otimes_{\cR_A^{I}}M^I$ inducing 
	$$\cR_A^{I \cap [s,1)}\otimes_{\cR_A^{[s,1)}}M^{[s,1)}\cong \cR_A^{I \cap [s,1)}\otimes_{\cR_A^I}M^I$$
	for any interval appearing in the cover above. In particular a coadmissible $\cR_A^{[r,1)}$-module is uniquely determined by its sections along an admissible cover.
\end{rem}
\begin{proof}We only give a sketch of the proof.
	Reordering the intervals and refining the cover allows us to assume without loss of generality $r=s$ and assume $r_0=r\leq r_1\leq s_0\leq r_2\leq s_1< \dots.$ In order to construct a coadmissible $\cR_A^{[r,1)}$-module we need to construct a compatible chain of $\cR_A^{[r,t_i]}$-modules, with $t_i$ converging to $1.$  We shall explain how to extend the module $M^{[r_0,s_0]}$ to a module $M^{[r_0,s_1]}$ satisfying $M^{I}\cong \cR_A^I \otimes_{\cR_A^{I \cap [r_0,s_1]}}M^{[r_0,s_1]}$ for $I=[r_0,s_0]$ or $I=[r_1,s_1].$ By assumption $M^{[r_0,s_0]}$ and $M^{[r_1,s_1]}$ can be glued along the isomorphism $$\cR_A^{[s_1,r_1]} \otimes_{\cR_A^{[r_0,s_0]}} M^{[r_0,s_0]}\cong \cR_A^{[s_1,r_1]} \otimes_{\cR_A^{[r_1,s_1]}} M^{[r_1,s_1]}$$ to a coherent $\cR_A^{[r_0,s_1]}$-module, which by Kiehl's theorem (cf. \cite[6.1 Theorem 4]{bosch2014lectures}) is (associated to) a finitely generated $\cR_A^{[r_0,s_1]}$-module, that we denote $M^{[r_0,s_1]}.$ By construction the sections along $I \in \{[r_0,s_0],[r_1,s_1]\}$ are $\cR_A^I\otimes_{\cR_A^I} M^{[r_0,s_1]}.$ Iterating this construction produces a sequence $M^{[r_0,s_n]}$ of $\cR_A^{[r_0,s_n]}$-modules and a compatible sequence of isomorphisms \linebreak $\cR_A^{[r_0,s_{n}]} \otimes_{\cR_A^{[r_0,s_{n+1}]}} M^{[r_0,s_{n+1}]} \cong M^{[r_0,s_{n}]}$ passing to global sections gives the desired result.
\end{proof} 

\begin{lem}
	\label{lem:coadmissibleuniform}
	Let $M^{[r,1)}$ be a coadmissible $\cR_A^{[r,1)}$-module and let \linebreak$\mathfrak{U}=\{[r_i,s_i],i \in \NN\}$ be an admissible cover of $[r,1).$
	\begin{enumerate}[(i)]
		\item $M^{[r,1)}$ is finitely generated if and only if there exists $n \in \NN$ independent of $I \in \mathfrak{U}$ such that each $M^I$ is generated by at most $n$ elements.
		\item $M^{[r,1)}$ is finitely presented if and only if there exist $(m,n) \in \NN^2$ independent of $I\in \mathfrak{U}$ such that each $M^I$ admits a presentation $$(\cR_A^I)^m \to (\cR_A^I)^{n} \to M^I \to 0.$$
		\item $M^{[r,1)}$ is finitely generated projective if there exists $n \in \NN$ independent of $I \in \mathfrak{U}$ such that each $M^I$ is generated by at most $n$ elements and each $M^I$ is flat over $\cR_A^I.$
	\end{enumerate}
\end{lem}
\begin{proof}
	See Proposition 2.1.13 in \cite{KPX}, whose proof works for any base field, and the subsequent Remark 2.1.14.
\end{proof}
We are now able to define $L$-analytic $(\varphi_L,\Gamma_L)$-modules. In the context of $\varphi_L$-modules over Robba rings we always assume $r_L<r_0.$ Recall $r_L^{1/q}=\lvert u \rvert$ for $0 \neq u \in LT[\pi_L],$ which ensures that $\varphi_L$ is well-defined on $\cR_L^{I}$ for any interval $I \subset [r_0,1).$
\begin{defn} Let $r_L<r_0.$
	A (projective) \textbf{$\varphi_L$-module} over $\mathcal{R}_A^{r_0}=\mathcal{R}_A^{[r_0,1)}$ is a finite (projective)  $\mathcal{R}^{r_0}_A$-module $M^{r_0}$ equipped with an isomorphism $\varphi_L^*M^{r_0}:=\cR_A^{r_0^{1/q}} \otimes_{\cR_A^{r_0},\varphi_L}M^{r_0} \cong M^{r_0^{1/q}}:=  \mathcal{R}_A^{r_0^{1/q}}\otimes_{\mathcal{R}_A^{r_0}}M^{r_0}.$ A $\varphi_L$-module $M$ over $\mathcal{R}_A$ is defined to be the base change of a  $\varphi_L$-module $M^{r_0}$ over some $\mathcal{R}_A^{r_0}.$ We call $M^{r_0}$ a \textbf{model} of $M$ over $[r_0,1).$
	A \textbf{$(\varphi_L,\Gamma_L)$-module} over the above rings is a $\varphi_L$-module whose model is projective and endowed with a semilinear continuous action of $\Gamma_L$ that commutes with $\varphi_L.$ Here continuous means that for every $m \in M^{[r,s]} :=  \mathcal{R}_A^{[r,s]}\otimes_{\mathcal{R}_A^{r_0}}M^{r_0}$ the orbit map $\Gamma_L \to M^{[r,s]}$ is continuous for the profinite topology on the left side and the Banach topology on the right-hand side.
\end{defn}
We will sometimes use the notation $M^{[r_0,1)}$ instead of $M^{r_0}$ for added clarity.
\begin{rem}
	A projective $\varphi_L$-module over $\mathcal{R}_A^{r_0}$ is coadmissible. Furthermore given a coadmissible $\cR^{r_0}$-module $M^{r_0}$ together with an isomorphism  $\varphi_L^*M^{r_0} \cong M^{r_0^{1/q}}$ it is projective if and only if it is flat.
\end{rem}
\begin{proof} The first statement follows from the fact that every finitely generated projective module is finitely presented and any finitely presented module is coadmissible. For every closed interval $[r,s]$ the sections $M^{[r,s]}$ are finitely generated over the Noetherian ring $\cR_A^{[r,s]}$ and hence finitely presented. The isomorphism $\varphi_L^*M^{r_0} \cong M^{r_0^{1/q}}$ restricts to an isomorphism $\varphi_L^*(M^I) \cong M^{I^{1/q}},$ which allows us to shift a given interval $I= [r_0,s_0]$ with $s_0>r_0^{1/q}$ and conclude that $M^{r_0}$ is uniformly finitely presented. Hence by \ref{lem:coadmissibleuniform}(iii) $M^{r_0}$ is projective if and only if it is flat.
\end{proof}
So far we worked with $\varphi_L$-modules over $\cR_A^{[r,1)}.$ In order to reduce the computation of cohomology to the level $[r,s]$ we introduce the notion of a $\varphi_L$-module over $[r,s].$ Contrary to the case $[r,1),$ where $\varphi(\cR_A^{[r,1)})$ can naturally be viewed as a subring of $\cR_A^{[r^{1/q},1)},$ the interval $[r,s]$ gets shifted to $[r^{1/q},s^{1/q}],$ which means that we can only compare a $\varphi_L$-module and its pullback after restricting to the overlap $[r^{1/q},s]$ assuming $s \geq r^{1/q}$. This makes the following definition rather artificial.
\begin{defn} Let $0<r<s<1$ with $s \geq r^{1/q}.$ A $\varphi_L$-module over $\cR_A^{[r,s]}$ is a finitely generated $\cR_A^{[r,s]}$ module $M^{[r,s]}$ together with an isomorphism
	$$\varphi_M^{\text{lin}}\colon \cR_A^{[r^{1/q},s]} \otimes_{\cR_A^{[r^{1/q},s^{1/q}]}} \varphi^*(M^{[r,s]}) \xrightarrow{\cong} M^{[r^{1/q},s]}:= \cR_A^{[r^{1/q},s]}\otimes_{\cR_A^{[r,s]}}M^{[r,s]}.$$ A morphism  $f: M^{[r,s]} \to M'^{[r,s]}$ of $\varphi_L$-modules is a $\cR_A^{[r,s]}$-linear morphism of the underlying modules such that the diagram 
	$$\begin{tikzcd}
		{\cR_A^{[r^{1/q},s]} \otimes_{\cR_A^{[r,s]}}M^{[r,s]}} \arrow[r, "\id \otimes f"] \arrow[d, "\cong"]                   & {\cR_A^{[r^{1/q},s]} \otimes_{\cR_A^{[r,s]}}M'^{[r,s]}} \arrow[d, "\cong"]     \\
		{\cR_A^{[r^{1/q},s]} \otimes_{\cR_A^{[r^{1/q},s^{1/q}]}}\varphi^*(M^{[r,s]})} \arrow[r, "\id \otimes (\id \otimes f)"] & {\cR_A^{[r^{1/q},s]} \otimes_{\cR_A^{[r^{1/q},s^{1/q}]}}\varphi^*(M'^{[r,s]})}
	\end{tikzcd}$$ commutes.
	We denote by $\varphi_M:M^{[r,s]} \to M^{[r^{1/q},s]}$ the semi-linear map induced by the isomorphism $\varphi_M^{\text{lin}}.$ When there is no possibility of confusion we simply write $\varphi_L$ instead of $\varphi_M.$
\end{defn}

\begin{lem}
	\label{lem:phieq} Let $r \in [0,1)$ and let $s \in (r^{1/q},1).$
	The functor that assigns to a projective $\varphi_L$-module over $\cR_A^{[r,1)}$ its section $M^{[r,s]}= \cR_A^{[r,s]}\otimes_{\cR_A^{[r,1)}}M^{[r,1)}$ is an exact equivalence of categories between projective $\varphi_L$-modules over $\cR_A^{[r,1)}$ and projective $\varphi_L$-modules over $\cR_A^{[r,s]}.$ 
\end{lem}
\begin{proof}
	We first show essential surjectivity.
	Let $M^{[r,s]}$ be a $\varphi_L$-module over $\cR_A^{[r,s]}.$ By assumption $\varphi^*(M^{[r,s]})=\cR_A^{[r^{1/q},s^{1/q}]} \otimes_{\cR_A^{[r,s]},\varphi_L} M^{[r,s]}$ is a finitely generated $\cR_A^{[r^{1/q},s^{1/q}]}$-module and we have an isomorphism $$\cR_A^{[r^{1/q},s]} \otimes_{\cR_A^{[r^{1/q},s^{1/q}]}} \varphi^*(M^{[r,s]}) \cong \cR_A^{[r^{1/q},s]}\otimes_{\cR_A^{[r,s]}}M^{[r,s]}=M^{[r^{1/q},s]}.$$ The right-hand side being the restriction of $M^{[r,s]}$ to $\operatorname{Sp}(A) \times_K \mathbb{B}^{[r^{1/q},s]},$ which is precisely the overlap $\operatorname{Sp}(A) \times_K \mathbb{B}^{[r,s]} \cap\operatorname{Sp}(A) \times_K \mathbb{B}^{[r^{1/q},s^{1/q}]}$ by our assumption on $s,$ and thus  allows us to glue $M^{[r,s]}$ and $\varphi^*(M^{[r,s]})$ to a coherent sheaf on $\operatorname{Sp}(A) \times_K \mathbb{B}^{[r,s^{1/q}]},$ which by Kiehl's theorem is associated to a finitely generated $\cR_A^{[r,s^{1/q}]}$-module that we denote by $M^{[r,s^{1/q}]}.$
	It remains to construct an isomorphism $$\cR_A^{[r^{1/q},s^{1/q}]}\otimes_{\cR_A^{[r^{1/q},s^{1/q^2}]}} \varphi^*(M^{[r,s^{1/q}]}) \cong \cR_A^{[r^{1/q},s^{1/q}]}\otimes_{\cR_A^{[r,s^{1/q}]}} M^{[r,s^{1/q}]}.$$
	Restricting $M^{[r,s^{1/q}]}$ to $[r^{1/q},s^{1/q}]$ gives us 
	\begin{equation}
		\label{eq:glueing}\varphi^*(M^{[r,s]}) \cong \cR_A^{[r^{1/q},s^{1/q}]} \otimes_{\cR_A^{[r,s^{1/q}]} } M^{[r,s^{1/q}]}
	\end{equation}
	by construction.
	To simplify notation let $I:=[r,s]$ and $J:=[r,s^{1/q}].$ Consider the commutative diagram 
	$$ \begin{tikzcd}
		\cR_A^{J^{1/q}} \arrow[r]              & \cR_A^{I^{1/q}}                                              \\
		\cR_A^J \arrow[u, "\varphi_L"] \arrow[r] & \cR_A^I  \arrow[u, "\varphi_L"]                         
	\end{tikzcd}$$ with the horizontal arrows being the natural maps. We can interpret restriction as a pullback along the canonical inclusion. The diagram tells us that the restriction to $I^{1/q}$ of the $\varphi_L$-pullback of $M^{J}$ is the $\varphi_L$-pullback of the restriction of $M^J$ to $I.$ In formulae
	$$\cR_A^{I^{1/q}} \otimes_{\cR_A^{J^{1/q}}}\varphi^*(M^J) = \varphi^*(M^I)$$
	and plugging in \eqref{eq:glueing} gives the desired $$\cR_A^{I^{1/q}} \otimes_{\cR_A^{J^{1/q}}}\varphi^*(M^J) \cong \cR^{I^{1/q}}\otimes_{\cR_A^{J}}M^J.$$
	Iterating this construction we obtain a $\varphi_L$-module over $\cR_A^{[r,1)}.$ Note that this module is finitely generated by \ref{lem:coadmissibleuniform} since each $M^{[r^{1/q^n},s^{1/q^n}]}$ is generated by the same number of elements as $M^{[r,s]}$ by construction. To conclude the projectivity of $M^{[r,1)}$ one uses \ref{lem:coadmissibleuniform}(iii). \\
	We next show that the functor is fully faithful. 
	Given a morphism $M^{[r,s]} \to N^{[r,s]}$ we can apply the previous construction to both modules at the same time and take as a morphism between their $\varphi$-pullbacks the $\varphi$-pullback of $f.$ These morphisms glue together due to the $\varphi$-compatibility of $f.$ If $f=f^{[r,s]}$ is the restriction of a morphism $f^{[r,1)}:M^{[r,1)}\to N^{[r,1)},$ then this construction yields a morphism $g^{[r,1)}$ such that $g-f=0$ on every $M^{[r^{q^{-k}},s^{q^{-k}}]}$ but then $g=f$ by the coadmissiblity of $M^{[r,1)}$ and the condition on $r,s$ that ensures that these intervals cover $[r,1).$ 
	Because $\cR_A^{[r,1)} \to \cR_A^{[r,s]}$ is flat, the functor $M^{[r,1)}\mapsto M^{[r,s]}$ is exact. Consider an exact sequence of projective $\varphi_L$-modules over $\cR_A^{[r,s]}$
	$$0 \to M_1^{[r,s]} \to M_2^{[r,s]} \to M_3^{[r,s]} \to 0.$$ Taking the pullback along $\varphi:\cR_A^{[r,s]} \to \cR_A^{[r^{1/q},s^{1/q}]}$ remains exact because the $M_i^{[r,s]}$ are flat. Hence the induced sequence $$0 \to M_1^{[r,1)}\to M_2^{[r,1)}\to M_3^{[r,1)}\to 0$$ is exact when restricted to each $[r^{q^{-k}},s^{q^{-k}}].$ Finally the global section sequence remains exact by Lemma \ref{lem:frechetstein}\textit{(iii)}.
\end{proof}
\begin{defn} \label{def:psi}
	Following \cite[Section 2]{SV15} we define $\psi_{col}\colon o_L \llbracket T\rrbracket \to o_L \llbracket T\rrbracket$ to be the unique $o_L$-linear endomorphism satisfying $$\varphi_L \circ \psi_{col}(f)(T) = \sum_{a \in LT[\pi_L]} f(a +_{LT} T)$$ for all $f \in o_L\llbracket T \rrbracket,$ where $+_{LT}$ denotes the addition via the Lubin-Tate group law. This endomorphism can be extended to a continuous endomorphism of the Robba ring $\cR_L$  (cf. \cite[Section 2.1]{FX12}), which we denote by the same symbol. We define $ \psi_{LT}:= \pi_L^{-1}\psi_{col}.$ We use the same symbol for the endomorphism $1 \otimes \psi_{LT}$ of $\cR_A = A\hat{\otimes}_L \cR_L.$ 
\end{defn}
Note that we have $\psi_{LT}\circ\varphi_L = \frac{q}{\pi}.$ In particular $\Psi := \frac{\pi}{q}\psi_{LT}$ is a continuous left-inverse of $\varphi_L.$
\begin{defn}Using the isomorphism $\varphi_M^*M^{r_0} \cong M^{r_0^{1/q}}$ we define $$\psi_M: M^{r_0^{1/q}} \cong \mathcal{R}_A^{r_0^{1/q}} \otimes_{\mathcal{R}_A^{r_0},\varphi_L}M^{r_0} \to M^{r_0}$$ by mapping $f \otimes m$ to $\psi_{LT}(f)m.$
	
\end{defn}
\begin{defn}
	A $(\varphi_L,\Gamma_L)$-module over $\cR_A$ is called \textbf{$L$-analytic} if its $\Gamma_L$-action is $L$-analytic.
\end{defn}
The following example is a sanity check.
\begin{rem} \label{ex:ringactionanalytic} The action of $\Gamma_L$ induced by $\gamma(T)= [\chi_{LT}(\gamma)](T)$ on the relative Robba ring
	$\cR_A= \varinjlim_{0\leq r<1}\varprojlim_{r<s<1}\cR_A^{[r,s]}$ is $L$-analytic.
\end{rem}
\begin{proof}
	Since $T \mapsto [\chi_{LT}(\gamma)](T)$ is an automorphism of each annulus $[r,s]$, this reduces to studying the action on $\cR_A^{[r,s]}= A \widehat{\otimes}\cR_K^{[r,s]}.$ Here the left tensor factor carries the trivial $\Gamma_L$-action and $\cR_K$ carries the usual action. By \cite[Proposition 2.3.4]{Berger} the latter is $L$-analytic . Since the trivial action is $L$-analytic the statement follows from \ref{lem:analytictensor}.
\end{proof}
\begin{lem} Let $A$ be $K$-affinoid and
	let $M$ be an $L$-analytic $(\varphi_L,\Gamma_L)$-module over $\cR_A.$ The $\Gamma_L$ action on $M$ extends uniquely to a separately continuous action of $D(\Gamma_L,K)$ satisfying $\delta_gm=gm$ and each morphism of $L$-analytic $(\varphi_L,\Gamma_L)$-modules is $D(\Gamma_L,K)$-equivariant.
\end{lem}
\begin{proof} After reducing to the corresponding statement for the Banach space $M^I$ for a closed interval this follows from the proof of \cite[Proposition 4.3.10]{SchneiderVenjakobRegulator}, which was proved for general $K$-Banach spaces. 
\end{proof}
\subsection{A standard estimate for $(\varphi_L,\Gamma_L)$-modules} A recurring theme in \cite{KPX} is the fact than upon restricting $M$ to a closed interval $I=[r,s]$ we have for the operator norm $$\lvert\lvert \gamma-1\rvert \rvert_{M^I}\xrightarrow{\gamma \to 1} 0.$$ This remains true for our case but in order to study the action of the distribution algebra via the operators $Z_n = \varphi^{n-n_0}(Z_{n_0}) \in D(\Gamma_n,K)$ we need to estimate the operator norm of these variables. Note that since $D(\Gamma_n,K)$ is a Fréchet space and $M^I$ is a Banach space the a priori separately continuous action $D(\Gamma_n,K) \times M^I$ is in fact jointly continuous.
\begin{rem}
	\label{rem:continuousoperator}
	The induced map $\rho: D(\Gamma_n,K) \to \cEnd_K(M^I),$ that maps $\lambda$ to the map mapping $x \to \lambda x,$ is continuous with respect to the operator norm on $\cEnd_K(M^I)$.
\end{rem}
\begin{proof} Let $\varepsilon>0.$
	Since the multiplication map $D(\Gamma_n,K) \hat{\otimes}_K M^I \to M^I$ is continuous with respect to the projective tensor product, there is a continuous semi-norm $p$ and a constant $c$ such that the ball $\{v \in D(\Gamma_n,K) \hat{\otimes}_K M^I, p\otimes\lvert -\rvert_I(v)\leq c\}$ is mapped into $\{m \in M^I \mid \lvert m\rvert\leq\varepsilon.\}$ If $\lambda \in D(\Gamma_n,K)$ satisfies $p(\lambda)\leq c,$ then $p\otimes\lvert -\rvert_I (\lambda \otimes n)\leq c$ for any $n\in M^I$ with $\lvert n\rvert_{M^I} \leq 1.$ In conclusion the ball $\{\lambda \in D(\Gamma_n,K) \mid p(\lambda)\leq c \}$ is mapped into $\{F \in \cEnd_K(M^I)\mid \lvert \lvert F\rvert \rvert_{M^I}\leq \varepsilon\},$ because for any $x \in M^I$ with $\lvert x\rvert \leq 1$ we have $\lvert \rho(\lambda)(x)\rvert \leq \varepsilon.$
\end{proof}
\begin{rem}
	\label{rem:phipowers}
	The sequence $T,\varphi(T),\varphi^2(T),\dots$ converges to $0$ with respect to the Fréchet-topology on $\cR_K^{[0,1)}.$
\end{rem}
\begin{proof}
	By \cite[Lemma 1.7.7]{Schneider2017} we have $\varphi^{2k}(T) \in T\pi_L^ko_L\llbracket T\rrbracket+T^ko_L\llbracket T\rrbracket.$ A small calculation further shows $\varphi^{2k+1}(T) \in T\pi_L^ko_L\llbracket T\rrbracket+T^ko_L\llbracket T\rrbracket.$
	Observe that the $r$-Gauß norm of any element in $T\pi_L^ko_L\llbracket T\rrbracket+T^ko_L\llbracket T\rrbracket$ is at most $\max(\lvert \pi_L\rvert^kr,r^k)$ and hence tends to $0$ for every $r \in (0,1).$
\end{proof}
\begin{lem}\label{lem:operatornormestimates}
	Let $M$ be a finitely generated module over $\cR_A^{r}$ with an $L$-analytic semi-linear $\Gamma_L$-action. Fix any closed interval $I=[r,s]\subset[r,1)$ and any Banach norm on $M^I.$   
	\begin{enumerate}[(i)]
		\item We have $\lvert\lvert \gamma-1\rvert \rvert_{M^I}\xrightarrow{\gamma \to 1} 0.$
		\item Furthermore $\lvert\lvert Z_n\rvert \rvert_{M^I}\xrightarrow{n \to \infty} 0.$
	\end{enumerate}
\end{lem}
\begin{proof}For the first statement let $\varepsilon>0$ and
	let $m_1,\dots,m_d$ be a system of generators of $M^I.$ We will show that there exists an open subgroup $U$ such that $\lvert\lvert \gamma-1 \rvert \rvert<\varepsilon$ for every $\gamma \in U.$ By the continuity of the action we have $\lim_{\gamma \to 1} (\gamma-1)m_i = 0$ for every $i$. Furthermore given $m= \sum_i f_im_i \in M^I$ we can treat each factor $f_im_i$ separately and get
	$$(\gamma-1)f_im_i = (\gamma-1)(f_i)m_i + \gamma(f_i)(\gamma-1)(m_i).$$
	We know that $(\gamma-1)(m_i)$ can be made arbitrarily small. It remains to show that $(\gamma-1)(f_i)$ converges to $0$ uniformly as $\gamma \to 1$ and that $\gamma(f_i)$ is bounded. 
	We have $\lvert \gamma(f_i) \rvert = \lvert \gamma(f_i)-f_i +f_i\rvert \leq \max(\lvert \gamma(f_i)-f_i\rvert,\lvert f_i\rvert).$ It thus suffices to show the statement for $M = \cR_A.$ Since $\cR_K\otimes_KA$ is dense in $\cR_A$ it suffices to check the corresponding statement there. The case $\cR_K$ is in \cite[Lemma 2.3.5]{Berger} hence we may find an open subgroup $U$ such that the result holds for $A=K$ and every $\gamma \in U.$ Let $v = \sum_i f_i \otimes a_i$ be some representation of $v \in \cR_K^I \otimes_K A.$ We get $(\gamma-1)v = \sum_i (\gamma-1)f_i \otimes a_i$  and thus $\lvert (\gamma-1)v \rvert \leq \max \varepsilon \lvert f_i\rvert \lvert a_i\rvert.$ Since this holds for any representation of $v$ we get $\lvert (\gamma-1)v\rvert \leq \varepsilon \lvert v\rvert,$ which proves the statement.
	For the second statement we combine Remark \ref{rem:phipowers} and Remark \ref{rem:continuousoperator} to conclude that given $\varepsilon >0$ there exists $k_0$ such that $\lvert\lvert \varphi^k(Z_{n_0}) \rvert\rvert_{M^I}<\varepsilon$ for any $k\geq k_0.$ Since $Z_n = \varphi^{n-n_0}(Z_{n_0})$ the conclusion follows.
\end{proof}
\begin{rem}
	In the classical case one works with the variable $\gamma-1$ and \ref{lem:operatornormestimates}(ii) is an immediate consequence of the continuity of the $\Gamma_L$ action since $\varphi_{cyc}(\gamma-1) = (1+(\gamma-1))^p-1=\gamma^p-1.$ 
\end{rem}

\section{Action on the Kernel of $\Psi.$}
Throughout this section let $M$ be an $L$-analytic $(\varphi_L,\Gamma_L)$-module over $\cR_A$ with a model over $\cR_A^{[r_0,1)}$ for some $r_0 \in (r_L,1).$
\subsection{Some technical preparation.}
\begin{lem}
	\label{lem:banachinvertierbar}
	Let $V$ be a $K$-Banach space and let $F,G \in \cEnd_K(V)$ such that $G$ is invertible and 
	$$\lvert\lvert F-G\rvert\rvert< \lvert \lvert G^{-1}\rvert\rvert^{-1}.$$
	Then $F$ is invertible.
\end{lem}
\begin{proof}
	By assumption the operator $$(1-FG^{-1})=-(F-G)G^{-1}$$ has operator norm $<1$ hence the series $$\sum_{k\geq 0} (1-FG^{-1})^k$$ converges to an inverse of $F\circ G^{-1}$ with respect to the operator norm. Using that $G$ is invertible we conclude that $F$ has to be invertible as well.
\end{proof}
\begin{lem}
	\label{lem:extendgeneral} Let $R$ be an $A$-Banach algebra, i.e. a complete normed $A$-Algebra. 
	Let $B$ be a $K$-Banach algebra and let $H: B \to R$ be a continuous $K$-algebra homomorphism. Then it extends to a continuous $A$-linear homomorphism 
	$$A \hat{\otimes}_KB \to R.$$
\end{lem}
\begin{proof}
	Let $a \in A$ and define a $K$-bilinear map $A \times B \to R$ by mapping $(a,b)$ to $aH(b).$ Using that $R$ is a topological $A$-module and the map $H$ is continuous one verifies that this map is separately continuous. Since $A$ and $B$ are Banach-spaces the inductive and projective tensor product topology agree and due to the completeness of $R$ we obtain an extension $$A \hat{\otimes}_K B \to R.$$ This extension is a $K$-algebra homomorphism because $\lambda H(b)=H(\lambda b)$ for any $b \in K$ and furthermore it is $A$-linear by construction. 
\end{proof}
\begin{rem}
	\label{rem:EndAclosed} Let $I \subset[r_0,1)$ be a closed interval. Then $\cEnd_A(M^I)$ is an $A$-Banach algebra
\end{rem}
\begin{proof} It suffices to prove that $\cEnd_A(M^I)$ is a closed subspace of the $K$-Banach algebra $\cEnd_K(M^I).$ For any $a \in A$ denote by the same symbol the multiplication-by-$a$-map. Then $\theta_a: \cEnd_K(M^I) \to \cEnd_K(M^I)$ mapping $f$ to $af-fa$ is continuous with respect to the operator norm and an endomorphism is $A$-linear if and only if it lies in the closed subspace $\bigcap_{a \in A}\ker(\theta_a).$
\end{proof}
\subsection{The group Robba ring and the structure of $M^{\psi=0}$}
A key observation is the following decomposition.
\begin{lem}
	\label{lem:colmezdecomp}
	We have $$\mathcal{R}_K^{[r,s]} \cong \bigoplus _{a \in o_L/\pi_L^n} \varphi_L^n\left(\mathcal{R}_K^{[r^{q^n},s^{q^n}]}\right)\eta(a,T).$$
\end{lem}
\begin{proof} See \cite[Proposition 1.4]{colmez2016representations}.
\end{proof}
We now define the relative Robba group ring $\cR_A(\Gamma_L).$
\begin{defn} \label{def:robbagroup}
	From the isomorphism $\chi_{LT}: \Gamma_L \to o_L^\times$ we get a canonical filtration $\Gamma_L = \Gamma_0 \supset \Gamma_1 \supset \Gamma_2 \supset \dots$ by defining $\Gamma_n:= \chi_{LT}^{-1}(1+\pi_L^no_L)$ for $n\geq1.$ For $n$ large enough we have an isomorphism $\Gamma_n \cong \pi_L^no_L$ by mapping $\gamma$ to $\log(\chi_{LT}(\gamma)).$ Let $n_0$ be minimal with this property. Define charts $l_n: \Gamma_n \xrightarrow{\log(\chi_{LT}(\cdot))} \pi_L^no_L \xrightarrow{\cong} o_L,$ where the second arrow is given by dividing by $\pi_L^n.$ This induces an isomorphim of Fréchet algebras $$D(o_L,K) \cong D(\Gamma_n,K).$$ Using the isomorphism from \ref{thm:Fourier} we can view the right-hand side as a ring of convergent power series in the variable $Z_n$ from \ref{def:variable}. By transport of structure we define the ring extensions $\cR_K(\Gamma_n) \cong \cR_K$ and $\cR_K^I(\Gamma_n) \cong \cR_K^I$ of $D(\Gamma_n,K),$ i.e., $\cR_K^I(\Gamma_n)$ is defined as the $I$-convergent power series in the variable $Z_n.$ 
\end{defn}	

To extend this definition to $n<n_0$ one first has to understand the interplay between inclusion of subgroups and radii of convergence. This requires some bookkeeping.
	Denoting by $i_{n+m}: \Gamma_{n+m} \to \Gamma_n$ the natural inclusion for $m \geq 0$ we get a commutative diagram
	$$\begin{tikzcd}
		o_L \arrow{r}{l_{n+m}^{-1}} \arrow[hook]{d}{\pi_L^m}  &\Gamma_{n+m} \arrow[hook]{d}{\iota_{n+m}} \\ o_L \arrow{r}{l_{n}^{-1}} &\Gamma_n,
	\end{tikzcd}$$ which induces a commutative diagram 
	$$
	\begin{tikzcd}
		D(o_L,K) \arrow{r}{(l_{n+m}^{-1})_*} \arrow[hook]{d}{(\pi_L^m)_*}  &D(\Gamma_{n+m},K) \arrow[hook]{d}{(\iota_{n+m})_*} \\ D(o_L,K) \arrow{r}{(l_{n}^{-1})_*} &D(\Gamma_n,K).
	\end{tikzcd}$$
	Taking the isomorphism $\cR^+ \cong D(o_L,K)$ as an identification and using the fact that $(\pi_L)_*$ corresponds to the map $\varphi_L$ we get a commutative diagram
	$$
	\begin{tikzcd}
		\cR_K^{I^{q^m}} \arrow{r}{(l_{n+m}^{-1})_*} \arrow[hook]{d}{\varphi_L^m}  &\cR_K^{I^{q^m}}(\Gamma_{n+m}) \arrow[hook]{d}{(\iota_{n+m})_*} \\ \cR_K^I \arrow{r}{(l_{n}^{-1})_*} &\cR_K^I(\Gamma_{n})
	\end{tikzcd}$$
	and by taking limits a corresponding diagram
	$$
	\begin{tikzcd}
		\cR_K \arrow{r}{(l_{n+m}^{-1})_*} \arrow[hook]{d}{\varphi_L^m}  &\cR_K(\Gamma_{n+m}) \arrow[hook]{d}{(\iota_{n+m})_*} \\ \cR_K \arrow{r}{(l_{n}^{-1})_*} &\cR_K(\Gamma_{n}).
	\end{tikzcd}$$
	By mapping $\gamma \in \Gamma_n$ to its Dirac distribution we obtain a canonical map $\Gamma_n \to D(\Gamma_n,K)^\times.$ By transport of structure from \ref{lem:colmezdecomp} we see that the natural maps induce topological\footnote{We endow the left-hand side with the maximum norm with respect to the decomposition $\ZZ[\Gamma_n] = \bigoplus_{\gamma \in \Gamma_n/\Gamma_{n+m}} \gamma \ZZ[\Gamma_{n+m}].$} isomorphisms $$\cR_K^{I^{q^m}}(\Gamma_{n+m}) \otimes_{\ZZ[\Gamma_{n+m}]} \ZZ[\Gamma_n] \to \cR_K^I(\Gamma_n)$$ and $$\cR_K(\Gamma_{n+m}) \otimes_{\ZZ[\Gamma_{n+m}]} \ZZ[\Gamma_n] \to \cR_K(\Gamma_n).$$ 
\begin{defn}
We extend definition \ref{def:robbagroup} to $0 \leq n \leq n_0$ by setting $$\cR_K^I(\Gamma_n):= \cR_K^{I^{q^{n_0-n}}}(\Gamma_{n_0}) \otimes_{\ZZ[\Gamma_{n_0}]} \ZZ[\Gamma_n]$$ and $$\cR_K(\Gamma_n):= \cR_K(\Gamma_{n_0}) \otimes_{\ZZ[\Gamma_{n_0}]} \ZZ[\Gamma_n],$$
	where the topology is given by the product topology with respect to the decomposition of $\ZZ[\Gamma_n].$
	Finally in the relative case we define $\cR_A^I(\Gamma_n)$ as the completed tensor product $\cR_K^I(\Gamma_n)\hat{\otimes}_K A$ endowed with the tensor product norm and $\cR_A(\Gamma_n)$ via $$\cR_A(\Gamma_n) := \varinjlim_{0\leq r<1}\varprojlim_{r<s<1}\cR_A^{[r,s]}(\Gamma_n).$$
	
\end{defn}
\begin{rem}
	\label{rem:RobbagroupFrechetStein} $\cR_A^{[r,1)}(\Gamma_n)$ is a Fréchet-Stein algebra.
\end{rem}
\begin{proof}
	If $n\geq n_0$ the statement is clear by transport of structure. If $n<n_0$ we recall the decomposition $$\cR_A^I(\Gamma_n)= \cR_A^{I^{q^{n_0-n}}}(\Gamma_{n_0}) \otimes_{\ZZ[\Gamma_{n_0}]} \ZZ[\Gamma_n].$$  Take a sequence $r=r_0,r_1,\dots$ converging to $1,$ let $I_k:=[r_0,r_k]$ and let $\tilde{I}_k:=I_k^{q^{n_0-n}}.$  We know that $\cR_A^{[r,1)^{q^{n_0-n}}}$ is Fréchet-Stein and hence the maps $\cR_A^{\tilde{I}_{k+1}} \to \cR_A^{\tilde{I}_{k}}$ are flat with dense image. Since $\ZZ[\Gamma_n]$ is free over $\ZZ[\Gamma_{n_0}]$ and hence flat these properties remain for the induced maps $$\cR_A^{\tilde{I}_{k+1}}(\Gamma_{n_0}) \otimes_{\ZZ[\Gamma_{n_0}]} \ZZ[\Gamma_n] \to \cR_A^{\tilde{I}_{k}}(\Gamma_{n_0}) \otimes_{\ZZ[\Gamma_{n_0}]} \ZZ[\Gamma_n].$$ The resulting rings are finite modules over the Noetherian rings $\cR_A^{\tilde{I}_{k+1}}(\Gamma_{n_0})$ hence themselves Noetherian.
\end{proof}
Note that we change the radius of convergence while also changing the group. This stems from the fact that, using suitable charts for $o_L \cong \Gamma_n \subset \Gamma_L$, the subgroup $\Gamma_{n+1}$ corresponds to the index $q$ subgroup $\pi_Lo_L$ and multiplication by $\pi_L$ corresponds to $\varphi_L$ via the Fourier isomorphism. This rule of thumb does not apply at the level $\Gamma_1 \subset \Gamma_L =\Gamma_0,$ where $[\Gamma_L:\Gamma_1]=q-1.$ In Proposition \ref{prop:kerpsizerl} it will become apparent, why this convention on radii makes sense for studying the action on $\ker(\psi).$ Another caveat, we would like to point out, is the fact that contrary to the cyclotomic case the notions of $r$-convergent distributions (introduced in \ref{def:rkonvergenteDist}) and ``$I$-convergent" distributions $\cR_A^I(\Gamma_L)$ are not related in an obvious way (outside of certain special cases). See \cite[Section 1.3]{Berger} for a precise description of the relationship.

\begin{lem}
	\label{lem:zer1}
For $r,s \geq r_0^{1/q^n}$ we have $$M^{[r,s]} \cong \bigoplus_{a \in o_L /\pi_L^n} \eta(a,T)\varphi_M^n{M^{[r^{q^n},s^{q^n}]}}.$$
	
\end{lem}
\begin{proof}
	Because the linearised map is an isomorphism we get 
	\begin{align*} M^{[r,s]} &\cong \mathcal{R}_A^{[r,s]} \otimes_{\mathcal{R}_A^{[r,s]}, \varphi_L^n} M^{[r^{q^n},s^{q^n}]} \\ &\cong  (\bigoplus _{a \in o_L/\pi_L^n} \varphi_L^n(\mathcal{R}^{[r^q,s^q]})\eta(a,T))\otimes_{\mathcal{R}_A^{[r^{q^n},s^{q^n}]},\varphi_L^n} M^{[r^{q^n},s^{q^n}]} \\ &\cong  \bigoplus_{a \in o_L /\pi_L^n} \eta(a,T)\varphi_M^n{M^{[r^{q^n},s^{q^n}]}}.
	\end{align*}
\end{proof}
In order to prove Theorem \ref{thm:kerpsiprojective} we will need a several base change formulae. These allow us, roughly speaking, to change the interval $[r,s]$ to an interval $[r,s]^{q^n},$ by replacing $M^{[r,s]}$ with $\varphi^nM^{[r^{q^n},s^{q^n}]}.$ 
\begin{prop}
	\label{prop:kerpsizerl}
	Let $I = [r,s]$ be an interval such that $\psi_M$ is defined on $M^{I^{1/q^n}}.$ We have
	$$(M^{I^{1/q^n}})^{\psi=0}\cong \ZZ[\Gamma_L] \otimes_{\ZZ[\Gamma_n]}\eta(1,T)\varphi_M^nM^I$$
	and 
	$$(M^{[r^{1/q^n},1)})^{\psi=0}\cong \ZZ[\Gamma_L] \otimes_{\ZZ[\Gamma_n]}\eta(1,T)\varphi_M^nM^{[r,1)}.$$
\end{prop}
\begin{proof}
	From \ref{lem:zer1} we get a decomposition 	$$M^{I^{1/q^n}}\cong \bigoplus_{a \in o_L /\pi_L^n}\eta(a,T)\varphi_M^nM^I.$$
	Recall that  $\frac{\pi_L}{q}\psi_{LT}(\eta(i,T)) = \eta(\frac{i}{\pi_L})$ if $i \in \pi_Lo_L$ and $0$ otherwise. Furthermore for $a \in o_L^\times \cong \Gamma_L$ we have $\eta(a,T) = \eta(\chi_{LT}(\chi_{LT}^{-1}(a)),T) = \chi_{LT}^{-1}(a) \eta(1,T).$
	Because $\chi_{LT}^{-1}(a) \in \Gamma_L$ induces an automorphism of $M^I$ and commutes with $\varphi_M$ we get $\chi_{LT}^{-1}(a)(\eta(1,T)\varphi^n_MM^I) = \eta(a,T)\varphi^nM^I.$ Combining everything we get \begin{align*}(M^{I^{1/q^n}})^{\psi=0}&\cong \bigoplus_{a \in( o_L/ \pi_L^no_L)^\times}\eta(a,T)\varphi^nM^I\\& = \bigoplus_{a \in (o_L/ \pi_L^no_L)^\times}\chi_{LT}^{-1}(a)(\eta(1,T)\varphi^nM^I)\\ &= \bigoplus_{\gamma \in \Gamma_L/\Gamma_n} \gamma(\eta(1,T)\varphi^nM^I) \\ \\&\cong \ZZ[\Gamma_L] \otimes_{\ZZ[\Gamma_n]}\eta(1,T)\varphi_M^nM^I. \end{align*}
	This proves the first formula. The second formula follows by passing to the limit $s \to 1.$ 
\end{proof}

These results show that we have to understand the $\Gamma_n$-action on \linebreak  $\eta(1,T)\varphi^n(M^I).$ Let $\gamma \in \Gamma_n$ such that $(\chi_{LT}(\gamma)-1)$ is divisible by $\pi_L^n.$ We compute 
$$\gamma(\eta(1,T)\varphi^n(m))  = \eta(1,T)\varphi^n\left(\eta\left(\frac{\chi_{LT}(\gamma)-1}{\pi_L^n},T\right)\gamma m\right).$$
We may thus equivalently study the action of $\Gamma_n$ on $M^I$ given by 
\begin{align}
	H_n:\Gamma_n &\to \cEnd_A(M^I) \\
	\gamma &\mapsto \left[m \mapsto \eta\left(\frac{\chi_{LT}(\gamma)-1}{\pi_L^n},T\right)\gamma m\right]
\end{align}
Leaning on the results of \cite{SchneiderVenjakobRegulator} we shall extend this action to an action of $A \hat{\otimes}_K D(\Gamma_n,K).$ Note that we deviate from their notation. We always use $T$ for the variable of $\cR_A$ acting via multiplication on $M$ and use $Z_n$ for the variable of $D(\Gamma_n,K)$ whenever $n \geq n_0$ acting on $M$ via continuous extension of the $K[\Gamma_L]$-action. 

\subsection{The action via $H_n.$}
We explain how to extend the $H_n$-action to $A \hat{\otimes}_K D(\Gamma_n,K).$ Since the $\Gamma_L$-action is $A$ linear it is natural to expect that this extension arises as a base change of an extension to $D(\Gamma_n,K).$ Such an extension is constructed in \cite[p. 69, Discussion after 4.3.13]{SchneiderVenjakobRegulator} for the case $A=K.$ If we forget the $A$-action and think of modules over $\cR_A$ by considering their underlying topological vector space we arrive at a similar situation. Strictly speaking the underlying $\cR_K$-module of a $(\varphi_L,\Gamma_L)$-module over $\cR_A$ is not a $(\varphi_L,\Gamma_L)$-module over $\cR_K$ since it is not finitely generated. For this reason the results of \cite{SchneiderVenjakobRegulator} are not directly applicable and we sketch the construction for the convenience of the reader.\\

In the following $I$ denotes one of the intervals $[r_0,r_0],[r_0,r_0^{1/q}].$ For  $m\geq 0$ set $\mathfrak{r_m}:= p^{{\frac{-1}{p^m}}}.$  We fix an integer $m_0$ such that for any $m\geq m_0$ we have $r_0^{1/q}<\mathfrak{r}_m$ and \begin{equation}\label{eq:estimateSV1}\lvert \eta(x,T)-1\rvert_I <\mathfrak{r}_m.
\end{equation}This is possible because the series has coefficients in $o_L$ and constant term $1.$   Furthermore Lemma \ref{lem:operatornormestimates} (i) allows us to choose $n_1\geq n_0$ such that
\begin{equation}\label{eq:estimateSV2} \lvert\lvert \gamma-1 \rvert \rvert_{M^I} < \mathfrak{r}_{m_0}.
\end{equation}
for any $\gamma \in \Gamma_{n_1}$ with respect to a fixed Banach-module norm on $M^I.$

\begin{prop}
	\label{prop:extensionSV}
	The $\Gamma_n$-action on $M^I$ via $H_n$ extends to a continuous ring homomorphism $$H_n: D_{\mathfrak{r}_m}(\Gamma_n,K) \to \cEnd_A(M^I)$$ for any $m \geq m_0$ and $n \geq n_1.$
\end{prop}
\begin{proof}
	We first construct an extension that we denote by the same symbol $$H_n: D_\mathfrak{r_m}(\Gamma_n,K) \to \cEnd_K(M^I),$$ which is induced by mapping a $\gamma-1$ to $H_n(\gamma)-1.$ 
	Since $M$ is assumed to be $L$-analytic and the map $\gamma \mapsto \pi_L^{-n}(\eta(\chi_{LT}(\gamma)-1),T)$ (as a function $\Gamma_n \to \cR_K^+$) is locally $L$-analytic one checks that the action via $H_n$ is $L$-analytic. It thus suffices to extend the action to an action of $D_\mathfrak{\QQ_p,r_m}(\Gamma_n,K)$ as the latter will factor as desired $$H_n: D_\mathfrak{\QQ_p,r_m}(\Gamma_n,K) \xrightarrow{\mathrm{can}} D_\mathfrak{r_m}(\Gamma_n,K) \to \cEnd_K(M^I).$$
	Let $\lambda   \in D_\mathfrak{\QQ_p,r_m}(\Gamma_n,K)$ and let $\mathbf{b} = (\gamma_1-1,\dots,\gamma_d-1),$ where $\gamma_1,\dots,\gamma_d$ is a $\ZZ_p$-basis of $\Gamma_n.$
	Recall that $\lambda$ admits a convergent expansion $$\lambda = \sum_{\mathbf{k} \in \NN_0^d} a_{\mathbf{k}}\mathbf{b}^{\mathbf{k}}.$$ We are reduced to showing that the operator $$H_n(\lambda) := \sum_{\mathbf{k} \in \NN_0^d} a_{\mathbf{k}}H_n(\mathbf{b}^{\mathbf{k}})$$ converges with respect to the operator norm on $M^I.$ Knowing that $\abs{a_\mathbf{k}}\mathfrak{r}_m^{\abs{\mathbf{k}}}$ tends to zero it suffices due to the sub-multiplicativity of the operator norm to show $\norm{H_n(\mathbf{b}^{\mathbf{k}})}_{M^I} \leq \mathfrak{r}_m$ for any $\mathbf{k}$ with $\abs{\mathbf{k}}=1,$ in particular, it suffices to show 
	$$\norm{H_n(\gamma-1)}_{M^I} \leq \mathfrak{r}_m$$ for any $\gamma \in \Gamma_m.$
	We write out $$H_n(\gamma-1)m = \eta\left(\frac{\chi_{LT}(\gamma)-1}{\pi_L^n},T\right)(\gamma m -m) +  (\eta\left(\frac{\chi_{LT}(\gamma)-1}{\pi_L^n},T\right)-1)m.$$
	The assumptions \eqref{eq:estimateSV1} and \eqref{eq:estimateSV2} assert that both summands are bounded above by $\mathfrak{r}_m \norm{m}_{M^I}.$ We conclude that the series defining $H_n(\lambda)$ converges with respect to the operator norm. Our proof also shows that $\lambda \mapsto H_n(\lambda)$ is bounded with operator norm bounded by $1,$ which shows that the extension we constructed is continuous.
	The assumption that $\gamma$ acts $\cR_A$-semi-linearly guarantees in particular that $\gamma-1$ acts $A$-linearly for any $\gamma \in \Gamma_L$ using \ref{rem:EndAclosed} we conclude that the image of this extension is contained in $\cEnd_A(M^I).$
\end{proof}
\begin{cor}
	\label{cor:basechangeaction}
	The $\Gamma_n$-action on $M^I$ via $H_n$ extends to a continuous ring homomorphism $$H_n: A \hat{\otimes}_K D_{\mathfrak{r}_m}(\Gamma_n,K) \to \cEnd_A(M^I)$$ for any $m \geq m_0$ and $n \geq n_1.$ Passing to the limit with respect to $m$ we obtain the desired extension
	$$H_n:A \hat{\otimes}_K D(\Gamma_n,K) \to \cEnd_A(M^I).$$
\end{cor}
\begin{proof} For the first part
	apply \ref{lem:extendgeneral} in conjunction with \ref{prop:extensionSV}. We have $D(\Gamma_n,K) = \varprojlim_m D_{\mathfrak{r}_m}(\Gamma_n,K)$ and using \ref{rem:Frechetlim} we conclude that $$A \hat{\otimes}_KD(\Gamma_n,K) \cong \varprojlim_m A \hat{\otimes} D_{\mathfrak{r}_m}(\Gamma_n,K).$$ 
\end{proof}
We may increase $m$ such that $\mathfrak{r}_m>r_0^{1/q}.$ In this case using \ref{lem:extendscalaraction} we can extend the scalar action of $\mathcal{O}_K(\mathbb{B})$ to an action of $D_{\QQ_p,\mathfrak{r}_m}(\Gamma_n,K)$ that we call scalar action via
$$\mathfrak{S}_n:D_{\QQ_p,\mathfrak{r}_m}(\Gamma_n,K) \xrightarrow{l_n*} D_{\QQ_p,\mathfrak{r}_m}(o_L,K) \xrightarrow{\operatorname{proj}} D_{\mathfrak{r}_m}(o_L,K) \xrightarrow{LT}\mathcal{O}_K(\mathbb{B}^I) .$$
If we denote by $Z_n$ a preimage of $T$ in $D(\Gamma_n,K)$ and by $X_n$ a lift to $D_{\QQ_p,\mathfrak{r}_m}(\Gamma_n,K)$ then by construction $\mathfrak{S}_n(X_n)$ acts as multiplication by $T$ on $M^I.$
Since $0 \notin I$ we know that the action of $T$ on $M^I$ is invertible and the goal is to compare the action of $H_n(X_n) = H_n(Z_n)$ with $\mathfrak{S}_n(X_n) = T.$ The following lemma allows us to choose a sequence of lifts whose $\mathfrak{r}_m$-norms do not depend on $n.$
\begin{lem}
	\label{lem:compatiblelifts}
	Fix a lift $X_{n_1}$ of $Z_{n_1}$ to $D_{\QQ_p,\mathfrak{r}_m}(\Gamma_{n_1},K).$ Then there exists a sequence $X_{n_1+l} \in D_{\QQ_p,\mathfrak{r}_m}(\Gamma_{n_1+l},K)$ such that $X_{n_1+l}$ is a lift of $Z_{n_1+l}\in D(\Gamma_{n_1+l},K)$ and $$\abs{X_{n_1+l}}_{D_{\QQ_p,\mathfrak{r}_m}(\Gamma_{n_1+l},K)} = \abs{X_{n_1}}_{D_{\QQ_p,\mathfrak{r}_m}(\Gamma_{n_1},K)}$$ for every $l\geq0.$ 
\end{lem}
\begin{proof}
	The charts satisfy $\l_{n}=\pi_L\l_{n+1}$ for every $n\geq n_0.$ Transporting the problem to $o_L$ and arguing inductively it suffices to show that given $\lambda \in D_{\QQ_p,\mathfrak{r}_m}(o_L,K)$ there exists an element $\tilde{\lambda} \in D_{\QQ_p,\mathfrak{r}_m}(\pi_Lo_L,K),$ whose projection to $D(\pi_Lo_L,K)$ is equal to ${\pi_L}_*(\overline{\lambda}),$ such that $\tilde{\lambda}$ satisfies $$\abs{\tilde{\lambda}}_{D_{\QQ_p,\mathfrak{r}_m}(\pi_Lo_L,K)} = \abs{\lambda}_{D_{\QQ_p,\mathfrak{r}_m}(o_L,K)}.$$ We claim that ${\pi_L}_*(\lambda)$ has the desired properties. Given a $\ZZ_p$-basis $b_1,\dots,b_d$ of $o_L$ the elements $\pi_Lb_1,\dots,\pi_Lb_d$ form a basis $\pi_Lo_L.$ Since the $\mathfrak{r}_m$-norm is independent of the choice of basis we see that the isomorphism $$D_{\QQ_p}(o_L,K) \to D_{\QQ_p}(\pi_Lo_L,K)$$ induced by the isomorphism $o_L \cong \pi_Lo_L$ given by multiplication-by-$\pi_L$ is isometric with respect to the respective $\mathfrak{r}_m$-norms hence extends to an isometric isomorphism of the respective completions.
\end{proof}

\begin{lem} 
	\label{lem:robbaextension}There exists $n_2\geq n_1$ such that
	the map $H_n$ constructed above extends to continuous ring homomorphism $$\cR_A^I(\Gamma_n) \to \cEnd_A(M^I)$$ for any $n\geq n_2.$
\end{lem}
\begin{proof} Lemma \ref{lem:compatiblelifts} allows us to fix a sequence of elements $X_n$ lifting $Z_n$ such that $C = \abs{X_n}_{D_{\QQ_p,\mathfrak{r}_m(\Gamma_n,K)}}$ is independent of $n\geq n_1.$
	Let $0<\varepsilon < \min (r_0/C,1).$ Having fixed such a sequence we assume that $n_2$ is large enough such that the following assumptions are satisfied: 
	\begin{enumerate}[label=\textbf{A.\arabic*}]
		\item \label{ass:a1}$\norm{\gamma-1}_{M^I}<\varepsilon r_0^{1/q}$ for every $\gamma \in \Gamma_{n_2}.$
		\item \label{ass:a2} Choose $l=l(\varepsilon)$ such that $\norm{[a]-\id}_{\cR_A^I}< \varepsilon$ for every $a \in 1+\pi_L^lo_L.$
		\item \label{ass:a3} $\frac{1+\pi_L^{n_2}x}{x\pi_L^{n_2}} = 1-\frac{\pi_L^{n_2}x}{2}+\dots$ belongs to $1+ \pi_L^lo_L.$
	\end{enumerate}	
	The first two conditions can be achieved by using \ref{lem:operatornormestimates} applied to $M^I$ and $\cR_A^I$ respectively. The third one can be achieved by making $n_2$ large enough after having chosen $l.$
	Let $n\geq n_2$ and fix a $\ZZ_p$-basis $\gamma_1,\dots,\gamma_d$ of $\Gamma_n.$ Write $$X_n = \sum_{\mathbf{k}\in \NN_0^d}a_{\mathbf{k}}(\delta_{\gamma_i}-1)^{\mathbf{k}}$$
	by construction we have $C = \sup_{\mathbf{k}} \abs{a_\mathbf{k}}\mathfrak{r}_m^{\abs{\mathbf{k}}}.$
	We claim $$\norm{H_n(X_n)-\mathfrak{S}_n(X_n)}_{M^I}<r_0 = \abs{T^{-1}}_I^{-1} \leq \norm{T^{-1}}_{M^I}^{-1}.$$
	We abbreviate $\alpha(\gamma) = l_n(\gamma)= \log(\chi_{LT}(\gamma))/\pi_L^n$ and denote by $\beta$ the chart $\beta(\gamma) = \frac{\chi_{LT}(\gamma)-1}{\pi_L^n}.$ By construction $$H_n(\gamma-1) = \eta(\beta(\gamma),T)(\gamma-1)+(\eta(\beta(\gamma),T)-1)$$
	and $$\mathfrak{S}_n(\gamma-1) = \eta(\alpha(\gamma),T)-1.$$
	We first show $$\norm{H_n(\gamma-1)^k-\mathfrak{S}_n(\gamma-1)^k}_{M^I}<\varepsilon ({r_0^{1/q}})^k.$$
	We have  $$\norm{H_n(\gamma-1)}_{M^I} \leq \sup (\norm{\eta(\beta(\gamma),T)(\gamma-1)}_{M^I}, \norm{\eta(\beta(\gamma),T)-1}_{M^I})<r_0^{1/q}$$
	and $$\norm{\eta(\alpha(\gamma),T)-1}_{M^I}\leq \abs{\eta(\alpha(\gamma),T)-1}_I <r_0^{1/q}$$ using that $\abs{\eta(a,T)-1}_I<r_0^{1/q}$ and $\abs{\eta(a,T)}_I=1$ for any $a \in o_L,$ the assumption \ref{ass:a1} together with $\varepsilon<1.$ This allows us to reduce the claim to the case $k=1$ by writing for $x = \gamma-1$
	\begin{align}
		\label{eq:estimatetrick}
		\nonumber &\mbox{\hspace{16pt}}H_n(x)^k-\mathfrak{S}_n(x)^k \\
		&= H_n(x)(H_n(x)^{k-1}-\mathfrak{S}_n(x)^{k-1})+ \mathfrak{S}_n(x)^{k-1}(H_n(x)-\mathfrak{S}_n(x)).
	\end{align}
	If $k=1$ we have 
	\begin{align}
		&H_n(\gamma-1)-\mathfrak{S}_n(\gamma-1) \\
		&= \eta(\beta(\gamma),T)(\gamma-1)+ (\eta(\alpha(\gamma),T)-1)-(\eta(\beta(\gamma),T)-1)\\
		&= \eta(\beta(\gamma),T)(\gamma-1)+ ([u(\gamma)]-1)(\eta(\alpha(\gamma),T-1).
	\end{align}
	Where $u(\gamma) = \beta(\gamma)/\alpha(\gamma),$ which due \ref{ass:a3} belongs to $1+\pi_L^lo_L$ using that $\eta(\beta(\gamma),T) = \eta(\alpha(\gamma)u(\gamma),T) = [u(\gamma)](\eta(\alpha(\gamma),T))$ and the fact that $1$ is fixed by the $\Gamma_L$-action.
	The assumptions \ref{ass:a1} and \ref{ass:a2} ensure that both terms can be estimated by $\varepsilon r_0^{1/q}.$
	Let $\mathbf{b} = (\gamma_1-1,\dots,\gamma_d-1).$ We next prove
	$$\norm{H_n(\mathbf{b})^{\mathbf{k}}-\mathfrak{S}_n(\mathbf{b})^\mathbf{k}}<\varepsilon (r_0^{1/q})^{\abs{\mathbf{k}}}$$
	for any multi-index $\mathbf{k} \in \NN_0^d$ by induction on the number $h$ of non-zero components of $\mathbf{k}.$ We already treated the case $h=1$ and may therefore split $\mathbf{k} = (k_1,0,\dots,0)+ (0,k_2,k_3,\dots,k_d)$ and assume that the corresponding estimate holds for $\mathbf{i} =   (k_1,0,\dots,0)$ and $\mathbf{j} = (0,k_2,k_3,\dots,k_d).$ Using the same trick as in \eqref{eq:estimatetrick} we rewrite 
	\begin{align*}
		&H_n(\mathbf{b})^\mathbf{k}-\mathfrak{S}_n(\mathbf{b})^\mathbf{k} \\
		&= H_n(\mathbf{b})^{\mathbf{j}}(H_n(\mathbf{b})^\mathbf{i}-\mathfrak{S}_n(\mathbf{b})^\mathbf{i})+ \mathfrak{S}_n(\mathbf{b})^\mathbf{i}(H_n(\mathbf{b})^\mathbf{j}-\mathfrak{S}_n(\mathbf{b})^\mathbf{j})
	\end{align*}
	and use the estimates $\norm{H_n(\mathbf{b})^{\mathbf{l}}}<(r_0^{1/q})^{\abs{\mathbf{l}}}$ (resp.$\norm{\mathfrak{S}_n(\mathbf{b})^{\mathbf{l}}}<(r_0^{1/q})^{\abs{\mathbf{l}}}$) that can be obtained in the same way as in the case $h=1.$ 
	Putting everything together we obtain the final estimate 
	\begin{align}
		\norm{H_n(X_n)-\mathfrak{S}_n(X_n)} &\leq \sup_{\mathbf{k}} \abs{a_{\mathbf{k}}} \norm{H_n(\mathbf{b})^{\mathbf{k}}-\mathfrak{S}_n(\mathbf{b})^\mathbf{k}}\\
		&< \varepsilon \sup_{\mathbf{k}} \abs{a_{\mathbf{k}}} (r_0^{1/q})^{\abs{\mathbf{k}}}\\
		\label{eq:estimateHnT}
		&< \varepsilon  \sup_{\mathbf{k}}  \abs{a_{\mathbf{k}}}\mathfrak{r}_m^{\abs{\mathbf{k}}} =\varepsilon C<r_0.
	\end{align}
	Using \ref{lem:banachinvertierbar} we conclude that $H_n(Z_n)$ is invertible and its inverse given by $$H_n(Z_n)^{-1}=T^{-1}((T^{-1}H_n(Z_n)-1)+1)^{-1} =T^{-1}\sum_{k\geq 0}(1-T^{-1}H_n(Z_n))^k$$ has operator norm $$\norm{H(Z_n)^{-1}}_{M^I}\leq \norm{T^{-1}}_{M^I}$$ and satisfies $$\norm{H_n(Z_n)^{-1}-T^{-1}}_{M^I}<\norm{T^{-1}}_{M^I},$$ which follows from the estimate \eqref{eq:estimateHnT}, which asserts that the expression in the geometric series has operator norm less than $1.$ From \eqref{eq:estimateHnT} and the strict triangle inequality we further conclude $$ \norm{H_n(Z_n)}_{M^I}= \norm {H_n(Z_n)-T+T}_{M^I} \leq \norm{T}_{M^I},$$ which means that given $f(T) \in \cR_A^I$ the operator $f(H_n(Z_n))$ converges to an operator on $M^I$ of operator norm bounded by $ \abs{f}_{I}.$ In particular we obtained the desired continuous homomorphism $$\cR_A^I(\Gamma_n) \to \cEnd_A(M^I)$$ given by mapping $Z_n$ to $H_n(Z_n).$ 
\end{proof}
\begin{rem} 
	\label{rem:sharpendestimates}
	Let $n\geq n_2$ and let $f \in \cR_A^I.$ 
	We have $$\norm{f(H_n(Z_n))-f(T)}_{M^I}<\norm{f(T)}_{I}.$$
\end{rem}
\begin{proof}
	We show that $$\norm{H_n(Z_n)^{\pm k}-T^{\pm k}}_{M^I}< \norm{T^{\pm k}}_{I}$$ holds for every $k \in \NN.$ The case $k=0$ is trivial and the case $k=1$ has been treated in the proof of \ref{lem:robbaextension}. We proceed inductively by expressing 
	\begin{align*}&H_n(Z)^{{\pm1}^k}-T^{{\pm1}^k} \\
		&= H_n(Z)^{{\pm1}(k-1)}(H_n(Z)^{\pm1}-T^{\pm1})+ T^{\pm1}(H_n(Z)^{{\pm1}(k-1)}-T^{{\pm1}(k-1)})\end{align*}
	and using the estimates $\norm{H_n(Z_n)^{\pm j}}_{M^I} \leq \norm{T^{\pm}}_{M^I}^j \leq \abs{T^{\pm}}_{I}^j = \abs{T^{\pm j}}_I$ for $j \in \{1,k-1\}$ that were obtained implicitly in the proof of $\ref{lem:robbaextension}.$ Note that $\abs{T^{\pm j}}_I = \abs{T^\pm}_I^j$ by definition of the $I$-norm.
\end{proof}

\begin{lem}
	\label{lem:projHn}
	$M^I$ is finite projective with respect to the $\cR_A^I(\Gamma_n)$-module structure induced by $H_n$ of the same rank as $M^I$ over $\cR_A^I.$  Any system of generators of $M^I$ as a $\cR_A^I$-module also generates $M^I$ as a $\cR_A^I(\Gamma_n)$-module (via $H_n$).
\end{lem}
\begin{proof}
	Choose $N^I$ (and a Banach norm on $N^I$) such that $M^I \oplus N^I = (\cR_A^I)^d= \bigoplus_{i=1}^d \cR_A^Ie_i$ endowed with the $\sup$-norm of the norms on $M^I$ and $N^I$. We endow $N^I$ with a tautological $\cR_A^I(\Gamma_n)$-module structure by letting the variable $Z=Z_n \in D(\Gamma_n,K)$ act as multiplication by $T \in \cR_A^I.$ Then the estimate from Remark \ref{rem:sharpendestimates} remains valid for $N^I,$ since $H_n(Z)-T$ acts as zero on $N^I$ by construction. Fix a basis of $M^I \oplus N^I$ and define
	\begin{align}
		\Phi: M^I\oplus N^I &\to (\cR_A^I)^d\\
		\sum f_i(T)v_i &\mapsto f_i(T)e_i 
	\end{align}
	and 
	\begin{align*}\Psi: (\cR^I_A)^d &\to (\cR^I_A(\Gamma_n))^d \to M^I \oplus N^I \\
		\sum f_i(T)e_i &\mapsto \sum f_i(Z_n)e_i \mapsto \sum f_i(H_n(Z))(v_i).
	\end{align*} 
	By construction $\Phi$ is a topological isomorphism and $\Psi \circ \Phi$ is an endomorphism of $M^I \oplus N^I$ leaving both $M^I$ and $N^I$ invariant. We claim
	$$\lvert \lvert  \Psi\circ \Phi-1\rvert \rvert_{M^I\oplus N^I}<1.$$ This implies that $\Psi \circ \Phi$ is an automorphism, but then $\Psi$ has to be an isomorphism. In particular the map $\cR_A^I(\Gamma_n)^d  \to M^I \oplus N^I$ has to be an isomorphism, which shows the projectivity and the second part of the statement.
	For the estimation observe that $\Psi\circ\Phi-1$ is $0$ on $N^I,$ hence we only need to concern ourselves with $M^I,$ where the estimate follows from \ref{rem:sharpendestimates}. Regarding the rank we compute
	\begin{align*}
		\operatorname{rank}_{\cR_A^I(\Gamma_n)}(M^I) &= \operatorname{rank}_{\cR_A^I(\Gamma_n)}(M^I\oplus N^I)-\operatorname{rank}_{\cR_A^I(\Gamma_n)}(N^I)\\
		&=\operatorname{rank}_{\cR_A^I}(M^I\oplus N^I)-\operatorname{rank}_{\cR_A^I}( N^I)\\
		&=\operatorname{rank}_{\cR_A^I}( M^I),
	\end{align*}
	using additivity of ranks in the first and third equation. The second equality follows by construction since on the one hand $M^I \oplus N^I$ is isomorphic to $\cR_A^I(\Gamma_n)^d$ and on the other hand $N^I$ is viewed as a $\cR_A^I(\Gamma_n)$-module via transport of structure along the isomorphism $\cR_A^I(\Gamma_n)\cong \cR_A^I.$ The statement about the generators follows from the fact that $\Psi\circ \Phi$ respects the decomposition $M^I \oplus N^I$ in the sense that $M^I$ (resp. $N^I$) is mapped into itself. Hence if $M^I$ admits a system of generators that lifts to a basis $v_i$ of $M^I \oplus N^I$ the statement becomes clear since we have shown that these form a basis of the $\cR_A^I(\Gamma_n)$-module $M^I \oplus N^I$ (with action via $H_n$). Having chosen some system of generators $(m_1,\dots,m_d)$ of $M^I$ we can always find a suitable $N^I$ such that the $m_i$ are projections of a basis of $(\cR_A^I)^d$ by taking the surjection $(\cR_A^I)^d \to M^I$ given by mapping $e_i \to m_i$ and splitting it using the projectivity of $M^I.$
\end{proof}
\subsection{Interpreting the results}
We now explain how these results translate to the original module. We found it convenient to introduce the following abstract notation.
\begin{defn}
	Let $I \in \{[r,s], [r,1)\}$ with $r \geq r_0$ and let $n \in \NN.$ We say that $M$ satisfies \textbf{property $\mathcal{P}(n,I)$ with respect to a system of generators $m_1,\dots,m_d \in M^I$} if the $A[\Gamma_n]$-action on  $\eta(1,T)\varphi^n(M^I)$ extends to an action of $\cR_A^I(\Gamma_n)$ with respect to which $M^I$ is projective and finitely generated by the elements $\eta(1,T)\varphi^n(m_i), i=1,\dots,d.$
\end{defn}
This notion depends on the choice of system of generators. Since we assumed that $M$ admits a model over $[r_0,1)$ for some (fixed) $r_0 \in (0,1)$ we may fix a system of generators $m_1,\dots,m_d$ of $M^{[r_0,1)}.$ For any $I \subset  [r_0,1)$ we take the images of $m_i$ as a choice of system of generators for $M^I.$ To keep notation simple in the following we refer to $\mathcal{P}(n,I)$ with respect to this choice of generating system.
\begin{prop}
	\label{prop:abstractproperties}
	Let $I \subset [r_0,1)$ be an interval, $l,n \in \NN$ and let $(I_k)_k$ be an admissible covering of $[r_0,1).$ Then 
	\begin{enumerate}
		\item[1.)] $\mathcal{P}(n+l,I)$ implies $\mathcal{P}(n,I^{1/q^l}).$
		\item[2.)] If $M$ satisfies $\mathcal{P}(n,I_k)$ for every $k$ then $M$ satisfies $\mathcal{P}(n,[r_0,1)).$
		
	\end{enumerate}
\end{prop}
\begin{proof}
	The first statement follows from the decomposition $$\eta(1,T)\varphi^n (M^{I^{1/q^l}})\ =\eta(1,T)\varphi^{n+l}(M^I) \otimes_{\ZZ[\Gamma_{n+l}]}\ZZ[\Gamma_n].$$
	For the second statement our assumptions guarantee that each $\eta(1,T)\varphi^n M^{I_k}$ is flat and finitely generated by at most $d$ elements. If $n \geq n_0$ such that $\cR_A^{[r_0,1)}(\Gamma_n)\cong \cR_A^{[r_0,1)}$ the statement follows from Lemma \ref{lem:coadmissibleuniform}. In the case $n<n_0$ one can adapt Lemma \ref{lem:coadmissibleuniform} since any covering of $[r_0,1)$ provides a system of algebras defining the Fréchet-Stein structure on $\cR_A^{[r_0,1)}(\Gamma_n)$ as explained in Remark \ref{rem:RobbagroupFrechetStein}.
\end{proof}
Our results so far translate as follows.
\begin{lem}
	\label{lem:coreargumentkerpsi}
	Let $I = [r_0,r_0]$ or $I=[r_0,r_0^{1/q}].$ Then there exists $n_1 \in \NN$ such that for any $n \geq n_1$ the property $\mathcal{P}(n,I)$ is satisfied.
\end{lem}
\begin{proof}
	This translates to the assertion of \ref{lem:projHn}.
\end{proof}
\begin{lem}
	\label{lem:gamma1action}
	There exists $r_2 \in [r_0,1)$ such that $M$ satisfies $\mathcal{P}(1,[r,1))$ for any $r\geq r_2.$ 
\end{lem}
\begin{proof}
	By Lemma \ref{lem:coreargumentkerpsi} we have $\mathcal{P}(n,I)$ for any $n\geq n_1.$	Applying Proposition \ref{prop:abstractproperties} 1.) we obtain $\mathcal{P}(n_1,I^{1/q^l})$ for any $l \geq 0.$ Notice that the intervals $I^{1/q^{l}}$ with $I$ as in Lemma \ref{lem:coreargumentkerpsi} cover $[r_0,1).$ Using Proposition \ref{prop:abstractproperties} 2.) we conclude that $\mathcal{P}(n_1,[r_0,1))$ is satisfied and applying Proposition \ref{prop:abstractproperties} 1.) yet again we conclude that $\mathcal{P}(1,[r_2,1))$ holds with $r_2 = r_0^{1/q^{n_1-1}}.$
\end{proof}
\begin{thm}
	\label{thm:kerpsiprojective}
	Let $M$ be an $L$-analytic $(\varphi_L,\Gamma_L)$-module over $\cR_A$ admitting a model over $[r_0,1),$ then there exists $r_1 \in [r_0,1)$ such that for any $r\geq r_1$ the $\Gamma_L$-action on $(M^{[r,1)})^{\psi=0}$ extends to an action of $\cR_A^{[r,1)}(\Gamma_L)$ with respect to which $(M^{[r,1)})^{\psi=0}$ is finite projective of rank $\operatorname{rank}_{\cR_A}(M).$ If $m_1,\dots,m_d$ generate $M^{[r,1)}$ then the elements $\eta(1,T)\varphi(m_1),\dots, \eta(1,T)\varphi(m_d)$ generate $(M^{[r,1)})^{\psi=0}$ as a $\cR_A^{[r,1)}(\Gamma_L)$-module. 
\end{thm}
\begin{proof}
	Using the decomposition $(M^{[r,1)})^{\psi=0} = \ZZ[\Gamma_L] \otimes_{\ZZ[\Gamma_1]} \eta(1,T)\varphi(M^{{[r^q,1)}})$ this follows from Lemma \ref{lem:gamma1action} by taking $r_1 = r_2^{1/q}.$
\end{proof}
We have implicitly proved the following result.
\begin{thm}
	\label{thm:Zinvertible}
	Let $M$ be an $L$-analytic $(\varphi_L,\Gamma_L)$-module over $\cR_A$ admitting a model over $[r_0,1).$ Let $n \geq n_0$ then there exists $r_1 \in [r_0,1)$ such that the action of $Z_n \in D(\Gamma_n,K)$ is invertible on $(M^{[r,1)})^{\psi=0}$ for any $r\in [r_1,1).$
\end{thm}
\begin{proof}
	We have seen that the action of $D(\Gamma_n,K)$ extends to an action of $\cR_A^{[r,1)}(\Gamma_n)$ $r \geq r_1$ with a suitable $r_1.$ Note that the variable $Z_n$ is a unit in every $\cR_A^J(\Gamma_n)$ for any interval $J \subset (0,1).$
\end{proof}
If $M$ is free we can further sharpen the results.
\begin{cor}
	\label{cor:freecase}
	Let $M$ be a free $L$-analytic $(\varphi_L,\Gamma_L)$-module with a model over $[r_0,1)$ such that $m_1,\dots,m_d$ are a basis of $M^{[r_0,1)}$. Then there exists $r_1 \in [r_0,1)$ such that the action of $\Gamma_L$ on $(M^{[r,1)})^{\psi=0}$ extends to a $\cR_A^{[r_1,1)}(\Gamma_L)$-action with respect to which $\eta(1,T)\varphi(m_1),\dots ,\eta(1,T)\varphi(m_d)$ form a basis.
\end{cor}
\begin{proof}
	We use the notation of the proof of Lemma \ref{lem:projHn}. Using that $M$ is free one can choose $N=0.$ This shows that $M^I$ is free over $\cR_A(\Gamma_n)$ with basis $\eta(1,T)\varphi^n(m_i).$ Tracing through the definitions we conclude that $\eta(1,T)\varphi(m_1),\dots,\eta(1,T)\varphi(m_d) $ are global sections of the projective and hence coadmissible $\cR_A^{[r,1)}(\Gamma_L)$-module $(M^{[r,1)})^{\psi=0}$ that form a basis of the module $(M^{J})^{\psi=0}$ for $J$ in a suitable cover of $[r,1).$ Then the map \linebreak  $\cR_A^{[r,1)}(\Gamma_L)^d \to (M^{[r,1)})^{\psi=0}$ mapping $e_i$ to $\eta(1,T)\varphi(m_i)$ is an isomorphism of coadmissible modules, hence the claim.
\end{proof}
\section{Analytic cohomology via Herr complexes}
In this chapter we introduce the analytic Herr complex, which serves as an analogue of the classical Herr complex with the operator $\gamma^{p^n}-1$ replaced by $Z_{n}$ and prove finiteness and base change properties similar to \cite{KPX}. In order to obtain the operators $Z_n$ on $M$ we need to choose an open subgroup of $\Gamma_L$ that is isomorphic to $o_L$, which in general is not a direct summand in $\Gamma_L.$ Before circumventing this difficulty we discuss the split case. 
\subsection{Analytic Herr complex for $e<p-1$}
Assume for the moment $e(L/\QQ_p)<p-1$\footnote{This also forces $p \neq 2.$}. Therefore we have 
$$\Gamma_L \cong o_L^\times \cong \kappa_L^\times \times U_1 \cong \text{torsion} \times o_L,$$ where the isomorphism $o_L \cong U_1$ is induced by $\exp(\pi_L\cdot).$
We denote by $\Delta \subset \Gamma_L$ the torsion subgroup.
\begin{defn}
	Let $M$ be an $L$-analytic $(\varphi_L,\Gamma_L)$-module over $\cR_A$ and let $f$ be an $A$-linear continuous operator that commutes with the action of $\Gamma_L.$ We define 
	$$C_{f,D(\Gamma_L,A)}(M):= [ 0 \xrightarrow{} M^\Delta \xrightarrow{(f-1,Z)} M^\Delta\oplus M^\Delta \xrightarrow{Z\oplus 1-f} M^\Delta \xrightarrow{} 0]$$
	concentrated in $[0,2].$ We denote by $H^{*}_{f,D(\Gamma_L,A)}(M)$ the cohomology of this complex.
\end{defn}
\begin{rem}
	\label{bem:herrquasi}
	The morphism 
	$$	\begin{tikzcd}
		M^\Delta \arrow[d,"{\operatorname{id}}"] \arrow[r ]      & M^\Delta \oplus M^\Delta \arrow[d, "-\frac{\pi_L}{q}\psi_{LT}\oplus\operatorname{id}"] \arrow[r] & M^\Delta \arrow[d, "-\frac{\pi_L}{q}\psi_{LT}"]   \\
		M^\Delta \arrow[r] & M^\Delta \oplus M^\Delta \arrow[r]                                          & M^\Delta                                                 
	\end{tikzcd}
	$$
	is a quasi-isomorphism between $C_{\varphi_L,D(\Gamma_L,A)}(M)$ and $C_{\frac{\pi_L}{q}\psi_{LT},D(\Gamma_L,A)}(M).$
\end{rem}
\begin{proof}
	The cokernel complex is $0$ because $\frac{\pi_L}{q}\psi_{LT}$ is surjective since it is the left inverse of $\varphi_L.$ The kernel complex is given by 
	$${M^{\Delta}}^{\psi_{LT}=0} \xrightarrow{Z} {M^{\Delta}}^{\psi_{LT}=0},$$ which is quasi-isomorphic to $0$ since by $\ref{thm:Zinvertible}$ the action of $Z$ is bijective on the kernel of $\psi_{LT}.$
\end{proof}
\subsection{The case of general $e.$}
Let $M$ be an $L$-analytic $(\varphi_L,\Gamma_L)$-module over $\cR_A.$
\begin{defn}
	Let $n \geq n_0$ such that $\chi_{LT}\circ \log$ induces an isomorphism $\Gamma_{n_0} \cong \pi_L^{n_0}o_L \cong o_L.$ We define
	$$C_{f,Z_n}(M):= [  M \xrightarrow{(f-1,Z_n)} M\oplus M \xrightarrow{Z_n\oplus 1-f} M].$$
	We denote the cohomology of this complex (concentrated in degrees $[0,2]$) by $H^i_{f,Z_n}(M).$
\end{defn} 
\begin{rem}
	The complexes $C_{\varphi_L,Z_n}(M)$ and $C_{\frac{\pi_L}{q}\psi_{LT},Z_n}(M)$ are quasi-isomorphic.
\end{rem}
\begin{proof}
	We may define the quasi-isomorphism analogously to $\ref{bem:herrquasi}$ and invoke Theorem \ref{thm:Zinvertible} to deduce that the action of $Z_n$ is invertible on the kernel of $\psi_{LT}.$ 
\end{proof}
\begin{defn}
	Let $m\geq n \geq n_0.$ We define the restriction $\res_{n,m}:C_{f,Z_n}(M) \to C_{f,Z_m}(M)$ as 
	$$	\begin{tikzcd}
		M\arrow[d,"{\operatorname{id}}"] \arrow[r ]      & M \oplus M \arrow[d, "\operatorname{id}\oplus \mathfrak{Q}_{m-n}"] \arrow[r] & M \arrow[d, "\mathfrak{Q}_{m-n}"]   \\
		M \arrow[r] & M \oplus M \arrow[r]                                          & M,                                                 
	\end{tikzcd}
	$$
	where $\mathfrak{Q}_{m-n}$ is defined using $Z_m = \varphi_L^{m-n}(Z_n)= \mathfrak{Q}_{m-n}Z_n.$ 
\end{defn}
\begin{lem}
	\label{lem:resbasic}
	Let $m \geq m'\geq n \geq n_0.$ Then 
	\begin{enumerate}
		\item We have $\res_{n,m} = \res_{m',m}\circ\res_{n,m'}.$
		\item For each $i$ we have $\operatorname{im}(\res_{n,m})(H^i_{f,Z_n}(M)) \subset (H^i_{f,Z_m}(M))^{Z_n=0}.$
	\end{enumerate}

\end{lem}
\begin{proof}
	The first statement follows by transport of structure from the corresponding computation in $\cR_K^+ \cong D(\Gamma_n,K).$ There we have for any pair $a,b \in \mathbb{N}$ $$\varphi_L^{a+b}(T) = \varphi_L^a(\varphi_L^b(T))= Q_a(\varphi_L^b(T)) \varphi_L^b(T).$$
	For the second statement we consider each degree individually.  $Z_m$ is divisible by $Z_n$ in the distribution algebra, which implies the statement in degree $0.$ In degree $1$ consider a class $\overline{(a,b)} \in H^1_{f,Z_n}$ with $Z_na=(f-1)b.$ We compute $$Z_n\res_{n,m}(a,b)=Z_n (a,\mathfrak{Q}_{m-n}b) = (Z_na,Z_mb) = (f-1b,Z_mb),$$ which lies in the image of the first differential of $C_{f,Z_m}(M).$
	In degree $2$ let $b \in M.$ Then $Z_n\mathfrak{Q}_{m-n}b = Z_mb = \partial(b,0) \equiv 0.$
\end{proof}
\begin{cor}
	Applying the above for $m=n$ shows that $H^i_{f,Z_n}(M)$ has trivial $\Gamma_n$-action.
\end{cor}
Contrary to the classical theory we can not define the analytic Herr complex for $\Gamma_L$ directly unless $e<p-1.$ Our observations so far show that the cohomology groups of the restricted Herr complex for any subgroup $\Gamma_n$ isomorphic to $o_L$ carry a residual $\Gamma_L/\Gamma_n$ action. This allows us to define the Herr cohomology for $\Gamma_L$ after choosing such a subgroup.
\begin{defn}Choose $n \geq n_0.$ We define 
	$H^i_{f,D(\Gamma_L,K)}(M):= (H^i_{f,Z_n}(M))^{\Gamma_L}.$ 
\end{defn}
\begin{lem} \label{lem:indep}
	$H^i_{f,D(\Gamma_L,K)}(M)$ is independent of the choice of $n \geq n_0.$
\end{lem}
\begin{proof}
	
	We are reduced to showing that $$\res_{n,m}:(H^i_{f,Z_n}(M)) \to (H^i_{f,Z_m}(M))^{Z_n=0}$$ is an isomorphism.
	The case $i=0$ follows from $H^0_{f,Z_n}(M) = M^{f=1,\Gamma_n}=(M^{f=1,\Gamma_m})^{\Gamma_n}.$
	For $i=1$ consider a pair $(a,b)$ with $Z_na = (f-1)b$ that is mapped to $0$ in $H^1_{f,Z_m}(M).$ Meaning that there exists a $v \in M$ satisfying $((f-1)v,Z_mv)=(a,\mathfrak{Q}_{m-n}b).$ We compute the image of $v$ in $H^1_{f,Z_n}(M)$ and obtain $((f-1)v,Z_nv)-(a,b)= (0,Z_nv-b),$ which is an element in the kernel of the multiplication-by-$\mathfrak{Q}_{m-n}$-map. But on the cohomology $Z_n$ acts as zero and hence $\mathfrak{Q}_{m-n}$ is invertible as it has a non-zero constant term, which proves injectivity. 
	For surjectivity consider $(a,b) \in M^2$ satisfying $Z_ma = (f-1)b.$ By the same argument as above the operator $\mathfrak{Q}_{m-n}$ is invertible on $H^1_{f,Z_m}(M)^{Z_n=0}$ and we may find a cocycle $(c,d)$ such that $\mathfrak{Q}_{m-n}(c,d) \equiv (a,b)$ in $H^1_{f,Z_m}(M).$ Then $(\mathfrak{Q}_{m-n}c,d)$ satisfies $Z_n(\mathfrak{Q_{m-n}}c) = Z_mc =(f-1)d$ and is mapped to the class of $(a,b)$ by the restriction map. \\
	It remains to treat the case $i=2.$  We first prove injectivity. Let $c \in M$ such that $\mathfrak{Q}_{m-n}c$ vanishes in $H^2_{f,Z_m}(M)$ i.e. $\mathfrak{Q}_{m-n}c \in Z_mM +(f-1)(M).$ This means $\mathfrak{Q}_{m-n}m \in Z_nM +(f-1)M$ and hence vanishes in $H^2_{f,Z_n}(M),$ but then $c$ already has to vanish in $H^2_{f,Z_n}(M).$ For surjectivity let $\overline{d} \in H^2_{f,Z_m}(M)^{Z_n=0}.$ By the preceding arguments we may find an element $\overline{c} \in H^2_{f,Z_m}(M)$ satisfying $\mathfrak{Q}_{m-n}(\overline{c})=\overline{d}.$ We can lift $\overline{c}$ to an element of $M$ and take its projection to $H^2_{f,Z_n}(M)$ in order to find a preimage of $d$ in $H^2_{f,Z_n}(M).$
\end{proof}
\begin{rem} Let $n \geq n_0.$
	The action of $\Gamma_L$ on $M$ induces a natural action on $C_{f,Z_n}(M),$ given by letting $\gamma \in \Gamma_L$ act in the usual way on each component in each degree. If $\gamma \in \Gamma_n$ then this action is homotopic to the identity. In other words, the image of $C_{f,Z_n}(M)$ in the derived category $\mathbf{D}(A)$ carries an action of $\Gamma_L/\Gamma_n.$
\end{rem}
\begin{proof} The action is well-defined because $f$ commutes with $\gamma \in \Gamma_L.$
	If $\gamma \in \Gamma_n$ then the action of $\gamma-\id$ is given by $\eta(a,Z_n)-1 = Z_nH(Z_n)$ with some $a \in o_L$ and $H(Z_n) \in o_L\llbracket Z_n\rrbracket.$ A small calculation shows that the maps 
	\begin{align*}
		M^2 &\to M\\
		(m,m') &\mapsto H(Z_n)m'
	\end{align*}
	and
	\begin{align*}
		M &\to M^2\\
		m &\mapsto (0,H(Z_n)m)
	\end{align*} 
	define a homotopy between $\gamma$ and $\id.$
\end{proof}
\subsection{Finiteness of analytic Herr cohomology} The results of the previous section allow us to flexibly change between the open subgroups $\Gamma_n$ used to define the analytic Herr complex and to simplify the notation we fix some $n\geq n_0$ and write $Z:=Z_n.$
The goal of this section is to prove that for any analytic $(\varphi_L,\Gamma_L)$-module $M$ over $\cR_A$ the cohomology groups $H^i_{\varphi_L,Z}(M)$ are finitely generated over $A$. We follow the strategy of \cite{bellovin2021cohomology} using the result from the previous chapter regarding the $Z$ action on the kernel of $\psi$ to arrive at a situation that allows us to apply results from \cite{kedlaya2016finiteness}. This approach differs from \cite{KPX} who first prove the finiteness of the Iwasawa cohomology of $M$ and compare it to the $(\varphi,\Gamma)$-cohomology of the cyclotomic deformation of $M.$ The finiteness of the cohomology of $M$ is obtained in \cite{KPX} as a corollary by writing $M$ as a base change of its deformation. Similar arguments to the ones of Bellovin already appear in \cite{kedlaya2018categories}. We denote by $\mathbf{D}^b_{perf}(A),\mathbf{D}^-_{perf}(A),\mathbf{D}_{perf}^{[a,b]}(A)$ the full subcategory of the derived category consisting of objects which are quasi-isomorphic to bounded (resp. bounded above, resp. concentrated in degree $[a,b]$) complexes of finite projective $A$-modules. 

\begin{defn}
	Let $M$ be an $L$-analytic $(\varphi_L,\Gamma_L)$-module over $\cR_A$ with model over $[r_0,1).$ For any $1>r\geq r_0$ we define 
	$$C_{\varphi_L,Z,[r,1)}(M):M^{[r,1)} \xrightarrow{(\varphi_L-1,Z)} M^{[r^{1/q},1)}\oplus M^{[r,1)} \xrightarrow{Z \oplus (1-\varphi_L)} M^{[r^{1/q},1)}.$$ For $r<s<1$ we define the complex
	$$C_{\varphi_L,Z,[r,s]}(M):M^{[r,s]} \xrightarrow{(\varphi_L-1,Z)} M^{[r^{1/q},s]}\oplus M^{[r,s]} \xrightarrow{Z \oplus (1-\varphi_L)} M^{[r^{1/q},s]}$$ for any $s \in [r^{1/q},1).$
	We obtain canonical morphisms of complexes (of $A$-modules)
	$$C_{\varphi_L,Z,[r,1)}(M) \to C_{\varphi_L,Z}(M)$$
	and 
	$$C_{\varphi_L,Z,[r,1)}(M) \to C_{\varphi_L,Z,[r,s]}(M).$$ We say that the cohomology of $M$ is computed on the level of $C_{\varphi_L,Z,[r,1)}(M)$ (resp. $C_{\varphi_L,Z,[r,s]}(M)$) if the first (resp. both) maps are quasi-isomorphisms.
\end{defn}

When working with $\varphi_L$-modules over $\cR_A$ we can identify the cohomology of the complex $[M \xrightarrow{\varphi_L-1} M]$ with the Yoneda extension groups in the category of modules over the twisted polynomial ring $\cR_A[X;\varphi_L]$ with $X$ acting as $\varphi_L$ on $\cR_A$ (cf. \cite[Definition 1.5.4]{kedlaya2013relative}). The analogous result holds for $\varphi_L$-modules over $[r,1)$ (resp. $[r,s]$) but requires some care, since $\varphi_L:M^{[r,1)} \to M^{[r^{1/q},1)}$ changes the ring over which the module is defined. 
\begin{defn}
	\label{def:phiinterval}
	Let $M$ be a $\varphi_L$-module over $\cR_A^I,$ where $I=[r,1)$ or $I=[r,s]$ with $s \in [r^{1/q},1).$ We denote by $\operatorname{Ext}^1_\varphi(\cR_A^I,M)$ the group of extensions (as $\varphi_L$-modules) of $\cR_A^I$ by $M.$
\end{defn}
\begin{rem}
	\label{rem:phiext0}
	Let $M$ be a $\varphi_L$-module over $\cR_A^I$ with $I$ as in \ref{def:phiinterval}. The natural map 
	
	\begin{align*}\ker(\varphi_M-1)&\to \operatorname{Hom}_\varphi(\cR_A^{I},M^{I})\\
		x &\mapsto (f \mapsto fx)
	\end{align*} is an isomorphism.
\end{rem}
\begin{proof}
	The map is well-defined because $\varphi_M(fx)=\varphi_L(f)\varphi_M(x)=\varphi_L(f)x$ and an inverse is given by mapping $\alpha \in \operatorname{Hom}_\varphi(\cR_A^{I},M^{I})$ to $\alpha(1).$
\end{proof}
\begin{lem}\label{lem:ext1phi}
	Let $M$ be a $\varphi_L$-module over $\cR_A^I,$ with $I$ as in \ref{def:phiinterval} and let $J:=I^{1/q}\cap I.$
	The map that assigns to $E\in \operatorname{Ext}^1_\varphi(\cR_A^I,M)$ the element $\varphi_E(e)-e\in M^{J},$ where $e$ is any preimage of $1 \in \cR_A^I,$ induces an isomorphism $$\operatorname{Ext}^1_\varphi(\cR_A^I,M) \cong M^{J}/(\varphi_M-1)(M^{I}).$$
\end{lem}
\begin{proof}
	One checks that $\varphi_E(e)-e$ is mapped to zero in $\cR_A^J$ and hence $\varphi_E(e)-e$ indeed belongs to $M^J.$ Let $\tilde{e}$ be another preimage of $1 \in \cR_A^I,$ then $e-\tilde{e}$ lies in $M^I$ and therefore $\varphi_E(e)-e=\varphi_E(\tilde{e})-\tilde{e} +\varphi_M(e-\tilde{e})-(e-\tilde{e}).$ This proves that the map in question is well-defined. Let $v \in  M^{J}/(\varphi_M-1)(M^{I})$ and define $E_v$ to have $M^I \times \cR_A^I$ as its underlying $\cR_A^I$-module with $\varphi_{E_v}(m,r):= (\varphi_M(m)+\varphi_L(r)v,\varphi_L(r)).$ The module $E_v$ is finitely generated over $\cR_A^I$ and $\varphi_{E_v}$ is $\varphi_L$-semi-linear. The linearised map is an isomorphism by the five lemma and hence $E_v \in \operatorname{Ext}^1_\varphi(\cR_A^I,M).$ The element $e_1:=(0,1)$ is a preimage of $1$ and satisfies $\varphi(e_1)-e_1 = (v,0),$ which proves the surjectivity. It remains to show injectivity. A computation shows that $E_v$ is the trivial extension if $v\equiv 0$ and hence it suffices to show that any extension $E$ is isomorphic to $E_v$ with $v = \varphi_E(e)-e.$ Since $\cR_A^{I}$ is free and hence projective any extension of $\cR_A^I$ by $M^I$ is split as a $\cR_A^I$-module. Let $s:\cR_A^I \to E$ be any $\cR_A^I$-linear section. And write $x \in E$ as $x=m+s(f) = m+fs(1).$ Then $\varphi_E(x)=\varphi_M(m)+\varphi_L(f)\varphi_E(s(1))= \varphi_M(m)+\varphi_L(f)(\varphi_E(s(1))-s(1)) +\varphi_L(f)s(1).$ In particular, $E$ is isomorphic to $E_v$ with $v= \varphi_E(s(1))-s(1).$
\end{proof}
\begin{lem}
	\label{lem:phiquasiiso}
	Let $M$ be a projective $\varphi_L$-module over $\cR_A$ with model over $[r,1)$ and let $s \in [r^{1/q},1).$ Then the canonical morphism $$[M^{[r,1)}\xrightarrow{\varphi_L-1} M^{[r^{1/q},1)}] \to[ M^{[r,s]}\xrightarrow{\varphi_L-1} M^{[r^{1/q},s]}] $$ induced by the restrictions $M^{[r,1)} \to M^{[r,s]}$ (resp. $M^{[r^{1/q},1)} \to M^{[r^{1/q},s]}$ ) is a quasi-isomorphism. 
\end{lem}
\begin{proof}
	The statement follows from \ref{lem:phieq} together with \ref{rem:phiext0} and \ref{lem:ext1phi}.
\end{proof}
\begin{lem}
	\label{lem:complexcomputecohomology}
	Let $M$ be an $L$-analytic $(\varphi_L,\Gamma_L)$-module over $\cR_A$ with model over $[r_0,1).$ Then there exists $1>r_1\geq r_0$ such that for any $1>r \geq r_1$ the cohomology $H^i_{\varphi_L,Z}(M)$ is computed by the complex 
	$$C_{\varphi_L,Z,[r,s]}:M^{[r,s]} \xrightarrow{(\varphi_L-1,Z)} M^{[r^{1/q},s]}\oplus M^{[r,s]} \xrightarrow{Z \oplus (1-\varphi_L)} M^{[r^{1/q},s]}$$ for any $s \in [r^{1/q},1).$
\end{lem}
\begin{proof}
	The complex computing the cohomology of $M$ is the direct limit of the corresponding complexes for $M^{[r,1)}$ as $r \to 1.$ By a cofinality argument we may take the colimit over ${r_0}^{1/q^n}$ with $n \to \infty.$ We hence need to show that the restriction maps (labelled as $\id$ below)
	$$	\begin{tikzcd}
		{C_{\varphi,Z,[r,1)}:} \arrow[d, "\id"] & {M^{[r,1)}} \arrow[r] \arrow[d, "\id"] & {M^{[r^{1/q},1)}\oplus M^{[r,1)}} \arrow[r] \arrow[d, "\id"] & {M^{[r^{1/q},1)}} \arrow[d, "\id"] \\
		{C_{\varphi,Z,[r^{1/q},1)}:}                & {M^{[r^{1/q},1)}} \arrow[r]                & {M^{[r^{1/q^2},1)}\oplus M^{[r^{1/q},1)}} \arrow[r]                    & {M^{[r^{1/q^2},1)}}                   
	\end{tikzcd}$$
	induce quasi-isomorphisms. Following \cite{bellovin2021cohomology} we do so in two steps. We show that the restriction above is homotopic to the map 
	$$\begin{tikzcd}
		{C_{\varphi,Z,[r,1)}:} \arrow[d, "\varphi_L"] & {M^{[r,1)}} \arrow[r] \arrow[d, "\varphi_L"] & {M^{[r^{1/q},1)}\oplus M^{[r,1)}} \arrow[r] \arrow[d, "\varphi_L"] & {M^{[r^{1/q},1)}} \arrow[d, "\varphi_L"] \\
		{C_{\varphi,Z,[r^{1/q},1)}:}                & {M^{[r^{1/q},1)}} \arrow[r]                & {M^{[r^{1/q^2},1)}\oplus M^{[r^{1/q},1)}} \arrow[r]                    & {M^{[r^{1/q^2},1)}}                   
	\end{tikzcd}$$
	and the latter induces the desired quasi-isomorphism. One can check that the maps $$\operatorname{pr}_1\colon M^{[r^{1/q},1)}\oplus M^{[r,1)} \to M^{[r^{1/q}]}$$ and $$(0,-\id)\colon M^{[r^{1/q},1)}\to M^{[r^{1/q},1)}\oplus M^{[r,1)}$$ induce a homotopy between $\varphi_L$ and $\id$. In order to see that $\varphi_L$ induces a quasi-isomorphism consider the left-inverse $\Psi:C_{\varphi_L,Z,[r^{1/q,1)})} \to C_{\varphi_L,Z,[r,1)}$ obtained by applying $\frac{\pi_L}{q}\psi_{LT}$ in each degree of the complex. By \ref{thm:Zinvertible} there exists $r_1$ such that for any $r \geq r_1$ the action of $Z$ on the kernel of $\Psi$ is invertible with continuous inverse (note that the constant $\frac{\pi_L}{q}$ does not change the kernel). In particular we obtain a decomposition $C_{\varphi_L,Z,[r^{1/q},1)} = \varphi_L(C_{\varphi_L,Z,[r,1)}) \oplus \ker(\Psi)$ as complexes of $A$-modules. We claim that the $(Z,\varphi_L)$-cohomology of the second summand vanishes, which implies the desired result. The vanishing of $H^0$ and $H^2$ is an immediate consequence of $Z$ being invertible on the kernel of $\psi_{LT}.$ We prove that $H^1_{\varphi_L,Z}((M^{[r,1)})^{\psi_{LT}=0})$ vanishes. Consider $(a,b)$ that satisfy $Za +(1-\varphi_L)b = 0$ then $a = Z^{-1}(\varphi_L-1)b=(\varphi_L-1)Z^{-1}b$ and obviously $b=ZZ^{-1}b.$ Setting $x:=Z^{-1}b$ we see that $(a,b) = ((\varphi_L-1)x,Zx)$ vanishes in $H^1.$ So far we proved that the cohomology is computed by the complex $C_{\varphi_L,Z,[r,1)}$ $r \geq r_1.$ Consider the canonical morphism $C_{\varphi_L,Z,[r,1)} \to C_{\varphi_L,Z,[r,s]}.$ Up to signs these complexes are the total complexes of the double complexes 
	$$\begin{tikzcd}
		M^{[r,1)} \arrow[r, "\varphi_L-1"] \arrow[d, "Z"] & M^{[r^{1/q},1)} \arrow[d, "-Z"] \\
		M^{[r,1)} \arrow[r, "\varphi_L-1"]                & M^{[r^{1/q},1)}               
	\end{tikzcd}$$
	(resp. for $M^{[r,s]}$)).
	By the Acyclic Assembly Lemma \cite[Lemma 2.7.3]{Weibel} together with \ref{lem:phiquasiiso} we conclude that the total complexes are quasi-isomorphic.
\end{proof}
\begin{defn}
	Let $R$ be a topological ring complete with respect to a submultiplicative semi-norm containing a topologically nilpotent unit. A map between two $R$-Banach modules $f\colon M \to N$ is called \textbf{completely continuous} if there exists a sequence of finitely generated $R$-submodules $N_i \subset N$ such that operator norms of $f_i\colon M \xrightarrow{f} N \to N/N_i$ converge to zero with respect to fixed Banach norms on $M,N$ and the quotient seminorm on $N/N_i.$
\end{defn}
\begin{lem}
	\label{lem:compcont}
	Let $R \to S$ be a bounded morphism of Banach-algebras over $A,$ that is completely continuous when viewing $R,S$ as Banach modules over $A.$ Let $M$ be a finitely generated $R$-Banach module such that $S \otimes_RM$ is an $S$-Banach module, then $$M \to M\otimes_RS$$ is completely continuous as a morphism of $A$-Banach modules.
\end{lem}
\begin{proof}
	See \cite[Remark 1.7]{kedlaya2016finiteness}.
\end{proof}
\begin{lem}
	
	Let $[r',s'] \subset [r,s] \subset (0,1)$ with $r'>r,$ $s'<s$ and $r,r',s,s' \in \lvert K^\times \rvert.$ The natural inclusions $$\cR_A^{[0,s]} \to \cR_A^{[0,s']}$$ and $$\cR_A^{[r,s]} \to \cR_A^{[r',s']}$$ are completely continuous.  \footnote{The assumption $r,s,r',s' \in \lvert K^\times\rvert$ can be weakened to $r,s,r',s' \in \sqrt{\lvert K^\times\rvert}$ without any difficulties. In our application $K$ carries a non-discrete valuation and cofinality arguments allow us to choose suitable intervals.}
\end{lem}
\begin{proof}
	The first case is proved analogously to the second case with a slightly simpler proof. By assumption there exist elements $\rho,\sigma,\rho',\sigma' \in K$ attaining the absolute values $r,s,r',s'$ respectively.   
	Write $\cR_A^{[r,s]} = A\langle T/\sigma,\rho/T \rangle$ and take as $N_i \subset \cR_A^{[r',s']}$ the subspace generated by the monomials $T^{-i},\dots, T^i.$ By expressing a series $f \in \cR_A^{[r,s]}$ as an element of $A\langle T/\sigma',\rho'/T \rangle$ and projecting modulo $N_i$ we obtain the estimate $$\lvert f\rvert_{A\langle T/\sigma',\rho'/T \rangle/N_i} \leq \max\left(\frac{r}{r'},\frac{s'}{s}\right)^{i+1} \lvert f\rvert_{A\langle T/\sigma,\rho/T \rangle}.$$ By our assumption on the intervals $\max\left(\frac{r}{r'},\frac{s'}{s}\right)<1$ and therefore the operator norms of the composed maps $\cR_A^{[r,s]} \to \cR_A^{[r',s']}/N_i$ tend to zero.
\end{proof}

\begin{lem}
	\label{lem:interval}
	Let $M^{[r,1)}$ be a $(\varphi_L,\Gamma_L)$-module over $\cR_A^{[r,1)}$ and let $0<r < r' \leq s'< s <1.$ Then the restriction $M^{[r,s]} \to M^{[r',s']}$ is completely continuous. 
\end{lem}
\begin{proof}
	By definition $M^{[r,s]}$ is a finitely generated module over the Banach algebra $\cR_A^{[r,s]}$ and $M^{[r',s']} = \cR_A^{[r',s']} \otimes_{\cR_A^{[r,s]}} M^{[r,s]}.$ The result follows from \ref{lem:compcont} because the natural inclusion $\cR_A^{[r,s]} \to \cR_A^{[r',s']}$ is completely continuous. 
\end{proof}
\begin{thm}
	\label{thm:cohfinite}
	Let $A$ be $K$-affinoid and let $M$ be an $L$-analytic $(\varphi_L,\Gamma_L)$-module over $\cR_A.$ Then the cohomology groups 
	$$H^i_{\varphi_L,Z}(M)$$ are finitely generated over $A.$
\end{thm}
\begin{proof}
	By Lemma \ref{lem:complexcomputecohomology} the cohomology can be computed on the level of $M^{[r,s]}$ for $r\geq r_1$ and any $s \in [r^{1/q},1).$ Choose any subinterval $[r',s'] \subset [r,s]$ like in Lemma \ref{lem:interval} satisfying in addition $s \in [r'^{1/q},1).$
	The restriction $M^{[r,s]} \to M^{[r',s']}$ induces completely continuous maps in each degree of $C_{\varphi_L,Z,[r,s]} \to C_{\varphi_L,Z,[r',s']}$ that are quasi-isomorphisms by Lemma \ref{lem:complexcomputecohomology}.  By Lemma 1.10 in \cite{kedlaya2016finiteness} the cohomology groups are contained in a finitely generated $A$-module and hence themselves finitely generated because $A$ is Noetherian.
	
\end{proof}
\begin{thm}
	\label{thm:perfect}
	Let $A,B$ be $K$-affinoid and let $M$ be an $L$-analytic $(\varphi_L,\Gamma_L)$-module over $\cR_A.$ Let $f\colon A  \to B$ be a morphism of $K$-affinoid algebras. Then:
	\begin{enumerate}[(1)]
		\item $C_{\varphi_L,Z}(M) \in \mathbf{D}^{[0,2]}_{\text{perf}}(A).$
		\item The natural morphism $C_{\varphi_L,Z}(M) \otimes_{A}^\mathbb{L} B \to C_{\varphi_L,Z}(M \hat{\otimes}_A B)$ is a quasi-isomorphism.
	\end{enumerate}
\end{thm}
\begin{proof}
	By \ref{thm:cohfinite} the cohomology groups of $C_{\varphi_L,Z}(M)$ are finitely generated and because $A$ is Noetherian quasi-isomorphic to a bounded above complex of finitely generated projective $A$-modules by \cite[\href{https://stacks.math.columbia.edu/tag/05T7}{Tag 05T7}]{stacks-project} using that the category of finitely generated modules over a Noetherian ring is abelian. By the analogue of \cite[Corollary 2.1.5]{KPX} $\cR_A$ is flat over $A$ and hence $C_{\varphi_L,Z}(M)$ consists of flat $A$-modules. Combining both points we see that $C_{\varphi_L,Z}(M)$ is quasi-isomorphic to a complex $$ X \to P_1 \to P_2,$$ where $P_1,P_2$ are projective and $X = \ker(P_1 \to P_2)$ is flat using \cite[Lemma 4.1.3]{KPX}. Then $X$ is a finitely generated submodule of a finitely generated module over a Noetherian ring and hence even finitely presented. We conclude that $X$ is finitely generated and projective by \cite[\href{https://stacks.math.columbia.edu/tag/00NX}{Tag 00NX}]{stacks-project} and therefore $C_{\varphi_L,Z}(M)\in \mathbf{D}^{[0,2]}_{\text{perf}}(A).$
	For the second statement the proof of \cite{KPX} carries over verbatim using \ref{lem:finitecomplete}.
\end{proof}
\begin{rem} \label{rem:specificinterval} More precisely \ref{thm:perfect}, \ref{thm:cohfinite} hold for $C_{\varphi_L,Z,I}(M)$ for any $I=[r,s]$ or $I=[r,1)$ with $r \geq r_1$ and $s \in [r^{1/q},1).$
\end{rem}
\begin{proof} The analogues of
	\ref{thm:cohfinite} and \ref{thm:perfect} (1) were proved implicitly. For \ref{thm:perfect} (2) the same proof works when we replace $M$ by $M^{I}$ the only subtlety being that in order to apply \cite[4.1.5]{KPX} we need $C_{\varphi_L,Z,I}(M\hat{\otimes}_AB) \in \mathbf{D}^{-}_{perf}(B),$ which we deduce by using  \ref{thm:cohfinite} requiring $r \geq r_1,$ with $r_1$ depending on $M$! One can check that the same $r_1$ works for $M \hat{\otimes}_A B$ using that the $(\varphi_L,\Gamma_L)$-action on $B$ is trivial.
\end{proof}
For the full analytic Herr cohomology we obtain a variant of \ref{thm:perfect}. Because the cohomology is defined by taking the invariants of the $(\varphi_L,Z)$-cohomology we can not formulate similar perfectness results (outside of the case $e<p-1$) and we only obtain a base change result in the flat case. 
\begin{rem}
	\label{rem:invariants}
	Let $R \to S$ be a flat morphism of commutative rings, $G$ be a finite group and $W$ an $R[G]$-module. Then 
	$(S\otimes_RW)^G = S\otimes_R W^G$
\end{rem}
\begin{proof}We can rewrite $W^G$ as $W^G =\operatorname{ker}(W\xrightarrow{\bigoplus(g-1)} \bigoplus_{g \in G}W)$ and apply the exact functor $S \otimes_R -.$
\end{proof}
\begin{cor}
	\label{kor:vollecohomo}
	Let $A,B$ be $K$-affinoid and let $M$ be an $L$-analytic $(\varphi_L,\Gamma_L)$-module over $\cR_A.$ Let $f\colon A  \to B$ be a flat morphism of $K$-affinoid algebras. Then:
	\begin{enumerate}[(1)]
		\item The groups $H^i_{\varphi_L,D(\Gamma_L,A)}(M)=H^i_{\varphi_L,Z}(M)^{\Gamma_L}$ are finitely generated and vanish for $i \neq 0,1,2.$
		\item The natural morphism $H^i_{\varphi_L,D(\Gamma_L,A)}(M) \otimes_A B \to H^i_{\varphi_L,D(\Gamma_L,A)}(M\hat{\otimes}_AB)$ is an isomorphism.
	\end{enumerate}
\end{cor}
\begin{proof}
	The first statement follows from \ref{thm:cohfinite} because $A$ is Noetherian. For the second statement we use \ref{thm:perfect} to conclude $$H^i_{\varphi_L,Z}(M)\otimes_A B = H^i_{\varphi_L,Z}(M\hat{\otimes}B)$$ and taking $\Gamma_L$-invariants we obtain $$H^i_{\varphi_L,D(\Gamma_L,A)}(M\hat{\otimes}_AB)=(B \otimes_A H^i_{\varphi_L,Z}(M))^{\Gamma_L},$$ using \ref{rem:invariants} and the fact that the $\Gamma_L$ action factors through $\Gamma_L/\Gamma_n$ we have
	$$(B \otimes_A H^i_{\varphi_L,Z}(M))^{\Gamma_L} = B \otimes_A H^i_{\varphi_L,Z}(M)^{\Gamma_L}=B\otimes_A H^i_{\varphi_L,D(\Gamma_L,A)}(M).$$
	
\end{proof}
\begin{cor}\label{cor:sheaves} Let $M$ be an $L$-analytic $(\varphi_L,\Gamma_L)$-module over $\cR_A.$
	Assigning to an affinoid subdomain $Sp(A') \subset Sp(A)$ the cohomology groups $H^i_{\varphi_L,Z}(M\hat{\otimes}_AA')$ (resp. $H^i_{\varphi_L,D(\Gamma_L,A)}(M\hat{\otimes}_AA'$)) defines coherent sheaves on $Sp(A).$
\end{cor}
\begin{proof} 
	By \cite[4/Corollary 5 p. 68]{bosch2014lectures} the map $A' \to A$ is flat.  Theorem \ref{thm:cohfinite} and the base change property \ref{thm:perfect}(2) assert that the sheaf is associated to the finitely generated module $H^i_{\varphi_L,Z}(M)$ and hence coherent. The second case is treated in the same way with the base change formula from Corollary \ref{kor:vollecohomo}.
\end{proof}
If $M$ is a $(\varphi_L,\Gamma_L)$-module over $\cR_L$ attached to an $L$-representation $V$ of $G_L$, one can relate analytic cohomology (defined via pro-analytic co-cycles as in section \ref{section:comparison}) to Galois cohomology. More specifically the first cohomology group is a subgroup of $H^1(G_L,V)$ parametrising the analytic extensions of $L$ by $V$. It does not follow formally that $H^1_{\varphi_L,D(\Gamma_L,K)}(K\hat{\otimes}_LM) \subset K \otimes H^1(G_L,V)$ as we do not know, whether the differentials of the complex computing analytic cohomology are strict. In the next section we will compare analytic cohomology to analytic extensions of $(\varphi_L,\Gamma_L)$-modules and analytic cohomology defined in terms of pro-analytic co-cycles. 
\section{Comparison to pro-analytic cohomology and description in terms of extension classes}
\label{section:comparison}
Let $M$ be an LF-space over $K$ with a pro-analytic action of an $L$-analytic monoid $G,$ i.e., a monoid object in the category of $L$-analytic manifolds.
We write $M$ suggestively as $M = \varinjlim_{0<r<1} \varprojlim_{r\leq s<1} M^{[r,s]}$ with Banach spaces $M^{[r,s]}.$  The pro-analytic cohomology groups are defined analogously to (semi)-group cohomology as the cohomology of the complex \linebreak $C^\bullet_\text{pro-an}(G,M)$ whose term in degree $n$ is given by $$C^\text{pro-an}(G^n,M):= \varinjlim_{0<r<1} \varprojlim_{r\leq s<1}C^{an}(G,M)$$ endowed with the usual inhomogeneous cochain differentials. We denote the pro-analytic cohomology groups by $H^i_\text{pro-an}(G,M).$
Our desired application in the context of $(\varphi_L,\Gamma_L)$-modules is the monoid $\varphi_L^{\NN_0}\times \Gamma_L,$ where $\varphi_L^{\NN_0}$ is given the structure of an $L$-analytic manifold with the discrete topology.
Using \cite[Theorem 11.6]{thomas2022cohomology} one can reinterpret this cohomology as a cone of $\varphi_L-1$ on a complex computing the pro-analytic cohomology of $\Gamma_L$ with coefficients in $M.$
\subsection{Comparison to $\operatorname{Ext}^1(\cR_A,M)$}
As before fix some $U=\Gamma_n$ for an $n\geq n_0$ and write $Z$ for the variable $Z_n \in D(\Gamma_L,K).$ We denote by  $\alpha(-):=\log(\chi_{LT}(-))/\pi_L^n$ the chart used to indentify $U$ with $o_L.$ For an $L$-analytic $(\varphi_L,\Gamma_L)$-module over $\cR_A$ we denote by $\operatorname{Ext}_{an}^1(\cR_A,M)$ the equivalence classes of extensions 
$$0 \to M \to E \to \cR_A \to 0$$ of $L$-analytic $(\varphi_L,\Gamma_L)$-modules. Given a preimage $e \in E$ of $1 \in \cR_A$ we can assign to it an element in $H^1_{an}(\varphi_L^{\NN_0}\times \Gamma_L,M)$ by taking the function $h \mapsto (h-1)e.$ This induces by \cite[Theorem 8.3]{thomas2022cohomology} an isomorphism $$\Theta_\text{pro-an}\colon \operatorname{Ext}_{an}^1(\cR_A,M) \xrightarrow{\cong} H^1_\text{pro-an}(\varphi_L^{\NN_0}\times \Gamma_L).$$
\begin{thm} \label{thm:h1vergleich} Let $M,E$ be $L$-analytic $(\varphi_L,\Gamma_L)$-modules over $\cR_A$ that fit in an exact sequence $$0 \to M \to E \to \cR_A \to 0.$$ Let $e$ be a preimage of $1$ in $E$ then $e \mapsto ((\varphi_L-1)e,Ze)$ gives a well-defined injection $$\Theta:\operatorname{Ext}_{an}^1(\cR_A,M) \to H^1_{\varphi_L,Z}(M),$$ whose image is contained in $H^1_{\varphi_L,D(\Gamma_L,K)}(M)$ inducing  a  bijection $$\operatorname{Ext}_{an}^1(\cR_A,M) \to H^1_{\varphi_L,D(\Gamma_L,K),M}(M).$$
\end{thm}
\begin{proof} Since all modules involved are $L$-analytic, the morphisms between them are even $D(\Gamma_L,K)$-linear.
	Because $1$ is invariant under the action of $\varphi_L$ and $U$ the element $((\varphi_L-1)e,Ze)$ lies in $M \times M$ and it is clear that $\partial_2((\varphi_L-1)e,Ze)=0.$ If $\tilde{e}$ is another preimage of $1$ then $e-\tilde{e}$ lies in $M$ and therefore $((\varphi_L-1)e,Ze)-((\varphi_L-1)\tilde{e},Z\tilde{e}) = ((\varphi_L-1)(e-\tilde{e}),Z(e-\tilde{e})) = \partial_1(e-\tilde{e}).$ Note that for any $\gamma\in \Gamma_L$ the element $\gamma e$ is yet another preimage of $1$ and the above computation shows that the cocycles $(\varphi_L-1)e,Ze)$ and $\gamma (\varphi_L-1)e,Ze)$ differ by a coboundary, which proves the $\Gamma_L$-invariance of $\overline{((\varphi_L-1)e,Ze))}.$
	For the injectivity let $E$ be an extension such that $((\varphi_L-1)e,Ze)$ vanishes in $H^1.$ Then there exists $d \in M$ such that $((\varphi_L-1)(e-d),Z(e-d))=0.$ That means $e-d \in E$ is a preimage of $1 \in \cR_A$ fixed by $\varphi_L$ and $U,$ which can be modified to be $(\varphi_L,\Gamma_L)$-invariant by replacing it with $\frac{1}{[\Gamma_L:U]}\sum_{\gamma \in \Gamma_L/U}\gamma(e-d)$ and therefore induces a section of $E \to \cR_A,$ which shows injectivity.
	For the surjectivity 
	let $\overline{(a,b)} \in H^1_{\varphi_L,D(\Gamma_L,K)}(M).$ We first explain how to construct an $L$-analytic coycle of the semigroup $\varphi_L^{\NN_0}\times\Gamma_L$ with values in $M.$ In the following we identify $\sigma$ with its corresponding Dirac distribution. We obtain $\sigma-1 = ZG_{\sigma}(Z)$ with a suitable power series $G_{\sigma} \in o_K\llbracket Z \rrbracket.$ Note that the map $\sigma \mapsto G_{\sigma}(Z)b$ defines a $1$-cocycle $c: U \to M$ as for $\tau \in U$ we have $$\sigma\tau-1 = \sigma(\tau-1) + (\sigma-1) = Z\delta_{\sigma}G_{\tau}(Z) + ZG_{\sigma}(Z).$$ We extend it to the whole group $\Gamma_L$ using $\frac{1}{[\Gamma_L:U]}$-times the corestriction (defined in \cite[Definition 2.1.2]{BFFanalytic}) such that the restriction of $c$ to $U$ is the cocycle we started with. Finally we define an extension of $M$ by $\cR_A$ by setting $E=M \times \cR_A$ as $\cR_A$-modules with actions $\sigma((m,r))=(\sigma m+ (\sigma r)c(\sigma),\sigma r )$ and $(\varphi_E(m,r) = (\varphi_M(m)+\varphi_L(r)a,\varphi_L(r))).$ In order to show that this extension is $L$-analytic we need to show that the function $\sigma \mapsto c(\sigma)$ is $L$-analytic. It suffices to show that for $m \in M^{[r,s]}$ for any interval $[r,s]$ and a sufficiently small open subgroup $U' \subset U$ the orbit map $\sigma \mapsto \sigma m$ restricted to $U'$ is $L$-analytic. Recall that $\sigma$ acts on $m$ via the operator $$\eta(\alpha(\sigma),Z)$$ for some fixed $n \in \mathbb{N}$ depending on $U$ and we abbreviate $x:= \alpha(\sigma).$ 
	We wish to show that the series 
	$$G_{\sigma}(Z)m = (\eta(x,Z)-1)/Z m = \sum_{k=1}^{\infty} \frac{(x\Omega \log_{LT}(Z))^k}{k!Z}m$$
	converges on some ball $\lvert x\rvert \leq \pi_L^j,$ which via the chart $\chi_{LT}$ corresponds to the desired subgroup $U'$.
Choose $j$ such that $\lvert\lvert \frac{(\Omega\log_{LT}(Z)\pi_L^j/Z)^k}{k!} \rvert \rvert_{M^{[r,s]}}$ converges to zero for $k \to \infty$. 
	The choice of $j$ is possible because the series $\exp(T)$ has non-trivial radius of convergence and the operator norm of $\log_{LT}(Z)/Z \in D(U,K)$ acting on $M^{[r,s]}$ is bounded. It remains to show that the image of $E$ is the original element $(\overline{a,b}).$ We may choose $e=(0,1)$ as an explicit preimage of $1.$ By construction $(\varphi-1)e = (a,0)$. It remains to show $Ze = (b,0).$ For $\sigma\in U$ we compute $(\sigma-1)(e) = (G_{\sigma}(Z)b,0).$
	Recall that the action of $D(\Gamma_L,K)$ is obtained by continuous extension of the $K[\Gamma_L]$ action and that the action of $Z \in D(U,K)$ agrees with the action of any $X \in D_{\QQ_p}(U,K)$ that projects to $Z.$ Choose such an element and express it as a convergent series $$X = \sum_{\mathbf{k} \in \NN_0^d} a_{\mathbf{k}}\mathbf{b}^\mathbf{k},$$ where $\mathbf{b}= (\gamma_i-1,\dots,\gamma_d-1)$ is a $\ZZ_p$-basis of $U.$ The series defining $X$ converges a fortiori in $D(U,K)$ by definition of the quotient topology and since $\gamma_i-1$ is given by the power series $\eta(\alpha(\gamma_i),Z)-1$ evaluating at $Z=0$ shows that necessarily $a_{0}=0.$ We set $G_i(Z):=G_{\gamma_i}(Z).$ Let $\mathbf{G}:=(G_1,\dots,G_d)$ such that the image of  $\mathbf{b}^\mathbf{k}$ under the projection $\operatorname{proj}:D_{\QQ_p}(U,K)\to D(U,K)$ is $(ZG_1(Z),\dots,ZG_d(Z))^\mathbf{k}.$
	We compute 
	\begin{align}  Z &= \operatorname{proj}(X) \\
		& = \sum_{0 \neq \mathbf{k} \in \NN_0^d} a_\mathbf{k} Z^{\abs{\mathbf{k}}}\mathbf{G}^\mathbf{k}\\
		& = Z (\sum_{0 \neq \mathbf{k} \in \NN_0^d} a_\mathbf{k} Z^{\abs{\mathbf{k}}-1}\mathbf{G}^\mathbf{k})	
	\end{align}
	We claim that the inner sum $\mu:= \sum_{0 \neq \mathbf{k} \in \NN_0^d} a_\mathbf{k} Z^{\abs{\mathbf{k}}-1}\mathbf{G}^\mathbf{k}$ converges with respect to the Fréchet topology to an element of $D(U,K)$ which satisfies $\mu Z =Z$ by construction and hence has to be $1$ since $D(U,K)$ is a domain. For the convergence we remark that the map $\lambda \mapsto Z\lambda$ is an injective continuous operator with closed image.
	In particular $ZD(U,K)$ is itself a Fréchet space and we conclude that $\lambda \mapsto \lambda Z$ is a continuous surjection $D(U,K) \to ZD(U,K)$ between Fréchet spaces and thus a homeomorphism by the open mapping theorem.
	We further compute for $\mathbf{k} \neq 0$
	$$\mathbf{b}^\mathbf{k}(e) = (Z^{\abs{\mathbf{k}}-1}\mathbf{G}^\mathbf{k}b,0),$$ which can be seen as follows. Without loss of generality assume $k_1 \neq 0.$ Since $U$ is commutative we may first apply $\gamma_1-1$ and obtain $(\gamma_1-1)(e) = (G_1(Z)b,0)$ by construction. The resulting element belongs to the image of $M$ under the natural inclusion $M \to E$ and hence $\gamma_i-1$ acts via multiplication by $ZG_i(Z).$ Putting everything together we conclude 
	$$Ze = \left(\sum_{0 \neq \mathbf{k} \in \NN_0^d} a_{\mathbf{k}}Z^{\abs{\mathbf{k}}-1}\mathbf{G}^\mathbf{k}b,0\right) = (\mu b,0) = (b,0).$$
\end{proof}
Implicitly we have shown a slightly stronger statement.

\begin{cor} The map $\Theta$ induces an  isomorphism $H^1_{\varphi_L,D(\Gamma_L,K)}(M) \cong H^1_\text{pro-an}(\varphi_L^{\NN_0}\times\Gamma_L,M)$ such that the following diagram commutes
$$
	\begin{tikzcd}
		{\operatorname{Ext}^1_{an}(\cR_A,M)} \arrow[r, "\Theta_\text{pro-an}"] \arrow[d, no head, equal] & {H^1_\text{pro-an}(\varphi_L^{\NN_0}\times\Gamma_L,M)} \arrow[d] \\
		{\operatorname{Ext}^1_{an}(\cR_A,M)} \arrow[r, "\Theta"]                                         & {H^1_{\varphi_L,D(\Gamma_L,K)}(M).}            
	\end{tikzcd}
$$
\end{cor}
\subsection{Comparison to Kohlhaaase's analytic cohomology}
\label{sec:kohl}
In \cite{Kohlhaase} Kohlhaase studies analytic (co-)homology theories for the category $\mathcal{M}_G$ of locally convex $K$-vector spaces with a separately continuous $D(G,K)$ action, where $K$ is a spherically complete field and $G$ is an $L$-analytic group. We can not make the assumption that $K$ is spherically complete.  Luckily, the relevant constructions from \cite{Kohlhaase} work for general (complete) fields.
Kohlhaase defines for $V \in \mathcal{M}_G$ cohomology groups $H^q_{an}(G,V)$, which can be computed via so-called $s$-projective resolutions of the trivial module $K$ or, if $V$ is barrelled\footnote{More generally it suffices for $V$ to be a hypomodule in the sense of loc.cit..}, by $s$-injective resolutions of $V.$ 
In loc.cit.\ the cohomology groups are further endowed with topologies of their own, but we are only interested in their algebraic properties. For $V,W \in \mathcal{M}_G$ we denote by $\mathcal{L}_G(V,W) \subset \mathcal{L}_b(V,W)$ the subspace of $D(G,K)$-linear morphisms. We use $\mathcal{L}_{L,b}(V,W)$ (resp. $\mathcal{L}_{K,b}(V,W))$ to denote $L$-linear (resp. $K$-linear) continuous maps endowed with the respective strong topology if confusion can arise.  Following Kohlhaase we call a map between locally convex topological vector spaces \textit{strong}, if both its kernel and cokernel admit closed complements. An exact sequence of locally convex topological vector spaces is called $s$-exact if it is exact and every map is strong. An object $P$ in $\mathcal{M}_G$ is called $s$-projective if $\mathcal{L}_G(P,-)$ takes $s$-exact sequences to exact sequence of abstract $K$-vector spaces. The definition of $s$-injective objects is not symmetrical (cf. \cite[Definition 2.12]{Kohlhaase}), but we will not need this notion. By an $s$-projective resolution of an object $V \in \mathcal{M}_H$ we mean an  $s$-exact sequence $\dots \to V_1 \to V_0 \to V\to 0,$ in which all terms (except for $V$) are $s$-projective. We write $V_\bullet$ for the complex that arises by omitting $V$. One can now define $H^q_{an}(G,V)$ to be the cohomology of $\mathcal{L}_G(Y_\bullet,V),$ where $Y_\bullet$ is an $s$-projective resolution of the trivial object $K.$ This notion is independent of the choice of $s$-projective resolution (cf. \cite[Remark 2.10]{Kohlhaase}).
Kohlhaase constructs  for $V \in \mathcal{M}_G$ a standard projective resolution of the form $B_{-1}(G,V) =V$ and $B_{q}(G,V) =D(G,K) \hat{\otimes}_{K,i}  B_{q-1}(G,V)$ for $q  \geq 0$ viewed as a $D(G,K)$-module via the left tensor factor with differentials given on elements of the form $x = \delta_{0} \otimes \dots \otimes \delta_q \otimes v$ (with Dirac distributions $\delta_i$) by 
$$d(x) = (-1)^{q+1}(\delta_0\otimes\dots\otimes \delta_qv) + \sum_{i=1}^q(-1)^i\delta_0 \otimes \dots \otimes\delta_{i-1}\delta_{i}\otimes\dots \otimes v $$
The construction of this resolution does not require the base field to be spherically complete.
\begin{lem} Let $V$ be a Hausdorff locally convex topological vector spaces over $K$ and let $W \in \mathcal{M}_G$ then the natural map
	$$\mathcal{L}_G(D(G,K)\hat{\otimes}_{K,i}V,W) \to \mathcal{L}_b(V,W)$$ induced by composing with the map $v \mapsto 1\otimes v$ is a continuous bijection.
\end{lem}
\begin{proof}
	See \cite[Lemma 2.2]{Kohlhaase} and note that spherical completeness of the base field is not required.
\end{proof} 
As a consequence of the above for $V=K$ the analytic cohomology of $W \in \mathcal{M}_G$ is computed by a complex of the form 
$$0 \to W \to \mathcal{L}_b(D(G,K),W) \to \mathcal{L}_b(D(G,K) \hat{\otimes}_{K,i}D(G,K),W) \to \dots $$
We will show that the complex above is (algebraically) isomorphic to the standard complex of pro-analytic cochains.
We warn the reader that this identification is not topological without further assumptions on $W.$
\begin{rem}
	Multiplication on $G$ induces a topological isomorphism $$D(G,K) \hat{\otimes}_{K,i}D(G,K) \cong D(G^2,K).$$ 
\end{rem}
\begin{proof}
	This follows from \cite[A.3]{schneider2005duality} applied over $L$ and completed base change to $K.$
\end{proof}
\begin{lem}\label{lem:FrtoLF}
	Let $V$ be a Fréchet space and let $W = \varinjlim_m W_m$ be an LF-space with Fréchet spaces  $W_m$. Then any morphism $V \to W$ factors over some $W_{m_0}.$ In particular the natural map $$\varinjlim_m \mathcal{L}_b(V,W_m)\to \mathcal{L}_b(V,W) $$ is an algebraic isomorphism.
\end{lem}
\begin{proof} See \cite[Corollary 8.9]{SchneiderNFA}.
\end{proof}
\begin{lem}
	\label{lem:TensorAdjunction} Let $K/L$ be a complete normed field extension, let $E$ be a Hausdorff locally convex $L$-vector space, and let $F$ be a  $K$-Banach space. Then the natural map
	$$\mathcal{L}_K(K \hat{\otimes}_{L,\pi} E, F) \to \mathcal{L}_L(E,F)$$ mapping a morphism $f\colon K \hat{\otimes}_{L,\pi} E \to  F$ to the map $f_E \colon e \mapsto f(1 \otimes e)$ is a topological isomorphism for the respective strong topologies.
\end{lem}
\begin{proof}
	By \cite[Theorem 10.3.9]{PGS} the map $E \to K \otimes_L E$ mapping $e$ to $1 \otimes e$ is a topological embedding and since $K$ and $E$ are Hausdorff the natural map $K \otimes_L E \to K \hat{\otimes}_L E$ is a topological embedding. In particular $f_E$ is continuous. Let $g \colon E \to F$ be an $L$-linear map. Clearly it extends uniquely to a $K$-linear map $g^K \colon K \otimes_L {E\to F}.$ It remains to see that $g^K$ is continuous with respect to the projective tensor product topology. For that purpose let $U \subset F$ be some open $o_K$-lattice. By continuity of $g$ we can find some open lattice $U' \subset E$ with $g(U') \subset U$ and we see that $o_K \otimes_{o_L} U' \subset K \otimes_LE$ is an open lattice satisfying $g^K(o_K \otimes_{o_L} U') \subset U.$ We conclude that $g^K$ extends uniquely to $K \hat{\otimes} E.$ By construction we have $(g^K)_E = g.$ For the other direction observe that $(f_E)^K$ and $f$ are both continuous and agree on the dense subset $K \otimes_LE.$ This shows $(f_E)^K =f$ and we have thus constructed an inverse to $f \mapsto f_E.$ It remains to see that the assignment is topological for the respective strong topologies. By \cite[Remark 10.4.2]{PGS}  for any bounded subset $B \subset E$ the closure $B^K$ of $o_K \otimes_{ o_L} B$ is bounded in $K \hat{\otimes}_L E$ and conversly for any bounded $B' \subset K \hat{\otimes}_L E$ we have that $B' \cap 1 \otimes E$ is bounded and, since $E \to 1 \otimes E$ is a homeomorphic embedding, can be identified with a bounded subset $B'_E$ of $E.$ Fix an open $o_K$-lattice $V \subset F.$ Our construction shows that $\{g\colon E \to F \mid g(B) \subset V \} $ is mapped into $\{g\colon K \hat{\otimes}_LE \to F \mid g(B^K) \subset V \}$ and the inverse map sends $\{f\colon K \hat{\otimes}_LE \to F \mid f(B') \subset V \}$ into $\{f\colon E \to F \mid f(B'_E) \subset V \}.$ We conclude that $f \mapsto f_E$ (resp. $g \mapsto g^K$) is continuous with respect to the strong topology. 
\end{proof}
The following Lemma is due to Venjakob.
\begin{lem} \label{lem:LFproan}Let $M = \varinjlim_r\varprojlim_{r<s<1}M^{[r,s]}$ be a $K$-LF-space. Then we have for $q \geq 0$ canonical algebraic isomorphisms
	\begin{align}
		\mathcal{L}_{b,K}(D(G^q,K),M) 
		&\cong \varinjlim_r\varprojlim_{r<s<1}\mathcal{L}_{b,K}(D(G^q,K),M^{[r,s]})  \nonumber\\
		&\cong \varinjlim_r\varprojlim_{r<s<1} \mathcal{L}_{b,L}(D(G^q,L),M^{[r,s]}) \nonumber\\
		&\cong  \varinjlim_r\varprojlim_{r<s<1} C^{an}(G^q,M^{[r,s]}) \nonumber\\
		&=C^\text{pro-an}(G^q,M) \nonumber.
	\end{align}
	
\end{lem}
\begin{proof}
	The first isomorphism is a consequence of Lemma \eqref{lem:FrtoLF} and the universal mapping property of projective limits. 	The second isomorphism follows from the adjunction from Lemma \ref{lem:TensorAdjunction}. The third isomorphism is \cite[Theorem 2.2]{schneiderteitlbaumlocallyanalytic} (applied over the spherically complete field $L,$ viewing $M^{[r,s]}$ as an $L$-Banach space by forgetting the $K$-linear structure). 
\end{proof}

A trace through the definitions shows that the composition of the above isomorphisms is compatible with the differentials on both sides. We obtain the following corollary, which is essentially a pro-analytic version of \cite[Remark 2.17]{Kohlhaase}.
\begin{cor}\label{cor:proan}
	The complexes $C^\bullet_\text{pro-an}(G,M)$ and $	\mathcal{L}_{b}(D(G^q,K),M) $ are canonically isomorphic as complexes of abstract vector spaces.
\end{cor}
For the group $G = o_L \cong U$ we obtain (by our general assumptions on $K$) that $D(U,K)$ is isomorphic to $\cR_K^{[0,1)}$ and we have an obvious $s$-projective resolution of $K$ given by
\begin{equation}\label{eq:proj}D(U,K) \xrightarrow{Z} D(U,K) \to K \to 0.\end{equation} Note that this resolution is also a projective resolution of $K$ as a $D(U,K)$-module and hence using  $\mathcal{L}_G(D(G,K),-) \cong \operatorname{Hom}_{D(G,K)}(D(G,K),-)$ we conclude $$H^q_{an}(U,V) \cong \operatorname{Ext}^q_{D(U,K)}(K,V).$$

\begin{cor} \label{cor:comparison} Let $M$ be an $L$-analytic $(\varphi_L,\Gamma_L)$-module over $\cR_A.$ Then there are canonical isomorphisms in $\mathbf{D}(A):$
	\begin{enumerate}[1.)]
		\item $[M \xrightarrow{Z}M] \simeq C^\bullet_\text{pro-an}(U,M)$
		\item $C_{\varphi_L,Z}(M) \simeq C^\bullet_\text{pro-an}(\varphi_L^{\NN_0}\times U,M).$
	\end{enumerate}
	
\end{cor}
\begin{proof} It follows formally that all terms in all involved complexes inherit from $M$ a natural $A$-module structure and all differentials turn out to be $A$-linear. The first quasi-isomorphism is obtained by comparing the  $s$-projective resolution \eqref{eq:proj} with Kohlhaase's standard resolution and applying  Lemma \ref{lem:LFproan} and Corollary \ref{cor:proan} for the LF-space $M.$
	The second part of the statement is obtained by taking the cone of $\varphi_L-1$ and using \cite[Theorem 11.6]{thomas2022cohomology}.
\end{proof}
\let\stdthebibliography\thebibliography
\let\stdendthebibliography\endthebibliography
\renewenvironment*{thebibliography}[1]{%
	\stdthebibliography{BGR84}}
{\stdendthebibliography}

\bibliographystyle{amsalpha}

\bibliography{Literatur}

\providecommand{\bysame}{\leavevmode\hbox to3em{\hrulefill}\thinspace}
\providecommand{\MR}{\relax\ifhmode\unskip\space\fi MR }
\providecommand{\MRhref}[2]{%
  \href{http://www.ams.org/mathscinet-getitem?mr=#1}{#2}
}
\providecommand{\href}[2]{#2}
\begin{thebibliography}{{Sta}21}

\bibitem[BC16]{bergercolmez2016theoriedesen}
Laurent Berger and Pierre Colmez, \emph{Th{\'e}orie de {Sen} et vecteurs
  localement analytiques}, Ann. Sci. {\'E}c. Norm. Sup{\'e}r.(4) \textbf{49}
  (2016), no.~4, 947--970.

\bibitem[Bel21]{bellovin2021cohomology}
Rebecca Bellovin, \emph{Cohomology of {$(\varphi,\Gamma)$}-modules over
  pseudorigid spaces}, arXiv preprint arXiv:2102.04820 (2021).

\bibitem[Ber16]{Berger2016}
Laurent Berger, \emph{{Multivariable ($\varphi, \Gamma$)-modules and locally
  analytic vectors}}, Duke Mathematical Journal \textbf{165} (2016),
  3567--3595.

\bibitem[BF17]{BFFanalytic}
Laurent Berger and Lionel Fourquaux, \emph{Iwasawa theory and {$F$}-analytic
  {Lubin-Tate} {$(\varphi,\Gamma)$}-modules}, Documenta Mathematica \textbf{22}
  (2017), 999--1030.

\bibitem[BGR84]{BGR}
Siegfried Bosch, Ulrich Güntzer, and Reinhold Remmert, \emph{Non-archimedean
  analysis : a systematic approach to rigid analytic geometry},
  Springer-Verlag, Berlin New York, 1984.

\bibitem[Bos14]{bosch2014lectures}
Siegfried Bosch, \emph{Lectures on formal and rigid geometry}, vol. 2105,
  Springer, 2014.

\bibitem[BSX20]{Berger}
Laurent Berger, Peter Schneider, and Bingyong Xie, \emph{{Rigid Character
  Groups, {Lubin-Tate} Theory}, and {$(\varphi, \Gamma)$}-{Modules}}, vol. 263,
  American Mathematical Society, 2020.

\bibitem[CC98]{cherbonnier1998representations}
Fr{\'e}d{\'e}ric Cherbonnier and Pierre Colmez, \emph{Repr{\'e}sentations
  p-adiques surconvergentes}, Inventiones mathematicae \textbf{133} (1998),
  no.~3, 581--611.

\bibitem[CC99]{cherbonnier1999theorie}
\bysame, \emph{Th{\'e}orie d’iwasawa des repr{\'e}sentations {$p$}-adiques
  d’un corps local}, Journal of the American Mathematical Society \textbf{12}
  (1999), no.~1, 241--268.

\bibitem[Col16]{colmez2016representations}
Pierre Colmez, \emph{Repr{\'e}sentations localement analytiques de
  {$GL_2(\mathbb{Q}_p)$} et ({$\varphi,\Gamma$)}-modules}, Representation
  Theory of the American Mathematical Society \textbf{20} (2016), no.~9,
  187--248.

\bibitem[Eme17]{emerton2017locally}
Matthew~J. Emerton, \emph{Locally analytic vectors in representations of
  locally {$p$}-adic analytic groups}, American Mathematical Soc., 2017.

\bibitem[FdL99]{Feaux}
Christian~Tobias F{\'e}aux~de Lacroix, \emph{Einige {Resultate} {\"u}ber die
  topologischen {Darstellung} {$p$}-adischer {Liegruppen} auf unendlich
  dimensionalen {Vektorr{\"a}umen} {\"u}ber einem {$p$}-adischen {K{\"o}rper}},
  Mathematisches Institut der Universit{\"a}t M{\"u}nster, 1999.

\bibitem[FX12]{FX12}
Lionel Fourquaux and Bingyong Xie, \emph{Triangulable {$o_F$}-analytic
  {$(\varphi,\Gamma)$}-modules of rank 2}, Algebra and Number Theory \textbf{7}
  (2012), no.~10.

\bibitem[Her98]{herr1998cohomologie}
Laurent Herr, \emph{Sur la cohomologie galoisienne des corps {$ p $}-adiques},
  Bulletin de la Soci{\'e}t{\'e} math{\'e}matique de France \textbf{126}
  (1998), no.~4, 563--600.

\bibitem[KL13]{kedlaya2013relative}
Kiran~S Kedlaya and Ruochuan Liu, \emph{Relative {$p$}-adic {Hodge} theory:
  foundations}, arXiv preprint arXiv:1301.0792 (2013).

\bibitem[KL16]{kedlaya2016finiteness}
\bysame, \emph{Finiteness of cohomology of local systems on rigid analytic
  spaces}, arXiv preprint arXiv:1611.06930 (2016).

\bibitem[Koh11]{Kohlhaase}
Jan Kohlhaase, \emph{The cohomology of locally analytic representations},
  Journal für die reine und angewandte Mathematik \textbf{2011} (2011),
  no.~651, 187--240.

\bibitem[KP18]{kedlaya2018categories}
Kiran Kedlaya and Jonathan Pottharst, \emph{On categories of {($\varphi$,
  $\Gamma$)}-modules}, Algebraic Geometry: Salt Lake City Part 2, Proc. Symp.
  Pure Math., Amer. Math. Soc., Providence \textbf{97} (2018), 281--304.

\bibitem[KPX14]{KPX}
Kiran Kedlaya, Jonathan Pottharst, and Liang Xiao, \emph{Cohomology of
  arithmetic families of {($\varphi,\Gamma$)}-modules}, Journal of the American
  Mathematical Society \textbf{27} (2014), no.~4, 1043--1115.

\bibitem[KV22]{kupferer2022herr}
Benjamin Kupferer and Otmar Venjakob, \emph{Herr-complexes in the {Lubin-Tate}
  setting}, Mathematika \textbf{68} (2022), no.~1, 74--147.

\bibitem[L{\"u}t16]{lutkebohmert2016rigid}
Werner L{\"u}tkebohmert, \emph{Rigid geometry of curves and their {Jacobians}},
  vol.~61, Springer, 2016.

\bibitem[PGS10]{PGS}
C.~Perez-Garcia and W.H. Schikhof, \emph{Locally convex spaces over
  non-archimedean valued fields}, Cambridge University Press, Cambridge, 2010.

\bibitem[Sch02]{SchneiderNFA}
Peter Schneider, \emph{Nonarchimedean functional analysis}, Springer Berlin
  Heidelberg, 2002.

\bibitem[Sch17]{Schneider2017}
Peter Schneider, \emph{Galois representations and {($\varphi,
  \Gamma$)}-modules}, Cambridge Studies in Advanced Mathematics, Cambridge
  University Press, 2017.

\bibitem[ST01]{schneider2001p}
Peter Schneider and Jeremy Teitelbaum, \emph{{$p$}-adic {Fourier} theory.},
  Documenta Mathematica \textbf{6} (2001), 447--481 (eng).

\bibitem[ST02]{schneiderteitlbaumlocallyanalytic}
\bysame, \emph{Locally analytic distributions and $p$-adic representation
  theory, with applications to {$GL_2$}}, Journal of the American Mathematical
  Society \textbf{15} (2002), no.~2, 443--468.

\bibitem[ST03]{schneider2003algebras}
\bysame, \emph{Algebras of p-adic distributions and admissible
  representations}, Inventiones mathematicae \textbf{153} (2003), no.~1,
  145--196.

\bibitem[ST05]{schneider2005duality}
\bysame, \emph{Duality for admissible locally analytic representations},
  Representation Theory of the American Mathematical Society \textbf{9} (2005),
  no.~10, 297--326.

\bibitem[{Sta}21]{stacks-project}
The {Stacks project authors}, \emph{The stacks project},
  \url{https://stacks.math.columbia.edu}, 2021.

\bibitem[Ste22a]{rusti}
Rustam Steingart, \emph{{Analytic cohomology of families of {$L$}-analytic
  {Lubin-Tate} {$(\varphi_L,\Gamma_L)$-modules}}}, Ph.D. thesis, Heidelberg
  University, 2022.

\bibitem[Ste22b]{RustIwasawa}
\bysame, \emph{Iwasawa cohomology of analytic
  {($\varphi_L,\Gamma_L)$}-modules}, arXiv preprint arXiv:2212.02275 (2022).

\bibitem[SV15]{SV15}
Peter Schneider and Otmar Venjakob, \emph{{Coates--Wiles homomorphisms and
  Iwasawa cohomology for Lubin--Tate extensions}}, Elliptic Curves, Modular
  Forms and Iwasawa Theory, Springer, 2015, pp.~401--468.

\bibitem[SV23]{SchneiderVenjakobRegulator}
\bysame, \emph{Reciprocity laws for {$(\varphi_L,\Gamma_L)$}-modules over
  {Lubin-Tate extensions}}, arXiv preprint arXiv:2301.11606 (2023).

\bibitem[Tho22]{thomas2022cohomology}
Oliver Thomas, \emph{Cohomology of topologised monoids}, Mathematische
  Zeitschrift (2022), 1--48.

\bibitem[Wei95]{Weibel}
Charles~A Weibel, \emph{An introduction to homological algebra}, no.~38,
  Cambridge university press, 1995.

\bibitem[Wit19]{Max}
Max Witzelsperger, \emph{Kategorienäquivalenz {$L$}-analytischer
  {Darstellungen} und {$(\varphi,\Gamma)$}-{Moduln} über dem {Robba-Ring}},
  Master's thesis, Heidelberg University (2019),
  \url{https://www.mathi.uni-heidelberg.de/~otmar/diplom/witzelsperger.pdf}.

\end{thebibliography}
\end{document}